\documentclass[11pt,leqno]{amsart}
\usepackage[dvipsnames]{xcolor}
\usepackage{amsmath}
\usepackage{amsfonts}
\usepackage{amssymb}
\usepackage{graphicx}
\usepackage{hyperref} 
\usepackage{a4wide} 
\usepackage{mathrsfs} 
\usepackage{relsize} 

\theoremstyle{plain}
\newtheorem{theorem}{Theorem}

\newtheorem{corollary}[theorem]{Corollary}
\newtheorem{definition}[theorem]{Definition}
\newtheorem{lemma}[theorem]{Lemma}
\newtheorem*{problem}{Problem}
\newtheorem{proposition}[theorem]{Proposition}

\theoremstyle{definition}
\newtheorem{example}[theorem]{Example}

\newtheorem{remark}[theorem]{Remark}

\numberwithin{equation}{section}
\numberwithin{theorem}{section}

\newcommand{\vol}{\mathrm{vol}}

\newcommand{\dive}{\operatorname{div}}
\renewcommand{\div}{\operatorname{div}}

\newcommand{\Hess}{\operatorname{Hess}}

\newcommand{\Lip}{\mathrm{Lip}}
\newcommand{\ZM}{\mathrm{ZM}}

\newcommand{\rr}{\mathbb{R}}
\renewcommand{\ss}{\mathbb{S}}

\newcommand{\nn}{\mathbb{N}}

\newcommand{\bb}{\mathbb{B}}
\newcommand{\hh}{\mathbb{H}}
\newcommand{\mm}{\mathbb{M}}

\newcommand{\riem}{\operatorname{Riem}}
\newcommand{\sect}{\operatorname{Sect}}
\newcommand{\ric}{\operatorname{Ric}}
\newcommand{\scal}{\mathrm{Scal}}

\newcommand{\supp}{\operatorname{supp}}
\newcommand{\tr}{\operatorname{trace}}

\newcommand{\dv}{\mathrm{dvol}}
\newcommand{\dx}{\mathrm{d}x}
\def\Xint#1{\mathchoice 
   {\XXint\displaystyle\textstyle{#1}}%
   {\XXint\textstyle\scriptstyle{#1}}%
   {\XXint\scriptstyle\scriptscriptstyle{#1}}%
   {\XXint\scriptscriptstyle\scriptscriptstyle{#1}}%
   \!\int}
\def\XXint#1#2#3{{\setbox0=\hbox{$#1{#2#3}{\int}$}
     \vcenter{\hbox{$#2#3$}}\kern-.5\wd0}}
\def\avint{\Xint-}

\newcommand{\inj}{\mathrm{inj}}
\newcommand{\harm}{\mathrm{harm}}
\newcommand{\euc}{\mathrm{Euc}}
\newcommand{\conj}{\mathrm{conj}}

\newcommand{\diam}{\mathrm{diam}}
\newcommand{\dist}{\mathrm{dist}}

\renewcommand{\a}{\alpha}
\renewcommand{\b}{\beta}

\renewcommand{\d}{\delta}
\newcommand{\e}{\epsilon}
\newcommand{\g}{\gamma}
\renewcommand{\l}{\lambda}

\newcommand{\s}{\sigma}
\newcommand{\vp}{\varphi}
\newcommand{\vr}{\varrho}
\newcommand{\w}{\omega}
\renewcommand{\O}{\Omega}

\newcommand{\bA}{\mathbf{A}}
\newcommand{\bH}{\mathbf{H}}
\newcommand{\bPsi}{\mathbf{\Psi}}
\newcommand{\by}{\mathbf{y}}

\newcommand{\bx}{\bar{x}}

\newcommand{\tW}{\widetilde{W}}
\newcommand{\tM}{\widetilde{M}}
\newcommand{\tB}{\widetilde{B}}
\newcommand{\tg}{\widetilde{g}}

\newcommand{\FM}{\frak{M}}

\newcommand{\CC}{\mathcal{C}}

\newcommand{\CM}{\mathcal{M}}
\newcommand{\CN}{\mathcal{N}}
\newcommand{\CH}{\mathcal{H}}

\newcommand{\CJ}{\mathcal{J}}

\newcommand{\A}{\mathscr{A}}
\newcommand{\D}{\mathscr{D}}
\newcommand{\B}{\mathscr{B}}
\newcommand{\C}{\mathscr{C}}
\renewcommand{\H}{\mathscr{H}}
\newcommand{\I}{\mathscr{I}}

\renewcommand{\L}{\mathscr{L}}
\newcommand{\M}{\mathscr{M}}

\newcommand{\R}{\mathscr{R}}

\begin{document}


\title[Calder\'on-Zygmund theory on manifolds]{Global Calder\'on-Zygmund inequalities on complete Riemannian manifolds}
\author{Stefano Pigola}
\address{Universit\`a degli Studi di Milano-Bicocca\\ Dipartimento di Matematica e Applicazioni \\ Via Cozzi 55, 20126 Milano - ITALY}
\email{stefano.pigola@unimib.it}
\date{September 28, 2021}
\begin{abstract}
 This paper is a survey of  some recent results on the validity and the failure of global $W^{2,p}$ regularity properties of smooth solutions of the Poisson equation $\Delta u = f$ on a complete Riemannian manifold $(M,g)$. We review different methods developed to obtain a-priori $L^p$-Hessian estimates of the form $\| \Hess(u) \|_{L^p} \leq C_1 \| u \|_{L^p} + C_2 \| f \|_{L^p}$ under various geometric conditions on $M$ both in the case of real valued functions and for manifold valued maps. We also present explicit and somewhat implicit counterexamples showing that, in general, this integral inequality may fail to hold even in the presence of a lower sectional curvature bound. The r\^ole of a gradient estimate of the form $\| \nabla u \|_{L^{p}} \leq C_1 \| u \|_{L^p} + C_2 \| f \|_{L^p}$, and its connections with the $L^{p}$-Hessian estimate, are also discussed.
\end{abstract}

\maketitle
\tableofcontents

\section{Introduction}
The so called {\it $L^{p}$-Calder\'on-Zygmund inequality}, $1<p<+\infty$, is one of the cornerstones in the regularity theory of  elliptic equations. One has a $W^{2,p}$ solution $u$ of the Poisson equation $L(u)  =f$ on a domain $\O \Subset \rr^{n}$ for some second order elliptic operator $L$ (with, say, smooth coefficients), and would like to estimate the $L^{p}$ norm of the full Hessian of $u$ in terms of the datum $f$ and, possibly, of the function itself. In fact, for any  $\O' \Subset \O$, one has
\begin{equation}\label{intro1}
\| \Hess(u) \|_{L^{p}(\O')} \leq A  \| u \|_{L^{p}(\O)} + B \| f \|_{L^{p}(\O)}
\end{equation}
where the constants $A,B>0$ do not depend on the given solution $u$ but, instead, they depend on $\dim \rr^{n}=n$ and $p$, on the geometry of the domains $\O$ and $\O'$ and on the structure of the operator $L$ in terms of the modulus of continuity of the coefficients and the ellipticity constant. It may happen that inequality \eqref{intro1} is still valid with $\O' = \O = \rr^{n}$ and, in this case, we speak of a {\it global} Calder\'on-Zygmund inequality for the operator $L$. The prototypical operator for which  this global phenomenon appears is represented by the Euclidean Laplacian $L=\Delta$. In fact, one has the more striking estimate
\begin{equation}\label{intro2}
\| \Hess(u) \|_{L^{p}} \leq A \| \Delta u \|_{L^{p}},
\end{equation}
for every $u \in C^{\infty}_{c}(\rr^{n})$ and for a universal constant $A=A(m,p)>0$. As a consequence, since, by Cauchy-Schwarz, it is always true that
\[
\| \Delta u \|_{L^{p}} \leq \sqrt{n} \| \Hess(u) \|_{L^{p}}
\]
and, by the {\it interpolation inequalities},
\[
\| \nabla u \|_{L^{p}} \leq C \{ \| \Hess (u) \|_{L^{p}} + \| u \|_{L^{p}} \}
\]
it follows the remarkable fact  that, in the Euclidean space, the Sobolev norms
\[
\| u \|_{W^{2,p}} = \| u \|_{L^{p}} + \| \nabla u \|_{L^{p}} + \| \Hess(u) \|_{L^{p}}
\]
and
\[
\| u \|_{\tW^{2,p}} = \| u \|_{L^{p}}  + \| \Delta u \|_{L^{p}}
\]
are equivalent on $C^{\infty}_{c}(\rr^{n})$. In particular, via the closure of $C^{\infty}_{c}(\rr^{n})$, they define the same (Banach) Sobolev space. This has interesting connections with the spectral theory of the operator $L=-\Delta+1$. \smallskip

One is naturally led to ask if it is possible to extend these global aspects of the Calder\'on-Zygmund theory to other differential operators. Keeping for the moment as a background the Euclidean space, and  writing $L = g^{ij}\partial^{2}_{ij}+ 1^{st}\text{ order terms}$, we could interpret the second order coefficient matrix $[g^{ij}]$ as (the inverse of) a Riemannian metric $g$ on $\rr^{n}$. From this point of view, even at the local level,  the constant $C$ reflects the geometry of $(\rr^{n},g)$ and this could prevent any extension to the global level. Whence, the possibility of producing a global Calder\'on-Zygmund inequality for the differential operator $L$ could be understood as a manifestation of the good geometry of $(\rr^{n},g)$\footnote{as a matter of fact, in this generality, the operator presents a drift that could introduce further nontrivial effects on the {\it metric measure space geometry and analysis}.} as in the case of the standard Euclidean space $(\rr^{n},g_{E})$. But in fact, there is no need to fix a Euclidean background space. Everything can be developed from the very beginning on a Riemannian manifold $(M,g)$ with the most natural (both analytically and geometrically) choice of the operator, namely the Laplace-Beltrami operator $L=\Delta$ of $M$. Clearly, if on the one hand, from the local viewpoint (i.e. on a relatively compact domain) there is no qualitative change with respect to the flat Euclidean setting, on the other hand the influence of the geometry encoded in the constants $A,B>0$ implies that, as alluded to above, switching from the local inequality to the global one is not at all a bypass product. \smallskip

Along this way we have approached a land where on a complete, non-compact Riemannian manifold $(M,g)$ it is given a global solution of the Poisson equation $\Delta u = f$ and we aim at deducing, a-priori  on the base of the geometry of $M$, that $|\Hess(u)| \in L^{p}(M)$ whenever $u,f \in L^{p}(M)$, and that the following a-priori estimate
\[
\| \Hess(u) \|_{L^{p}(M)} \leq A  \| u \|_{L^{p}(M)} + B \| f \|_{L^{p}(M)}
\]
holds for some constants $A,B>0$ depending only on $\dim M,p$ and on the geometry of $M$. Beside this, we also ask for a companion global gradient estimate of the form
\[
\| \nabla u \|_{L^{p}(M)} \leq  A  \| u \|_{L^{p}(M)} + B \| f \|_{L^{p}(M)}
\]
so to conclude that, in fact,
\[
\| u \|_{W^{2,p}(M)} \leq  A  \| u \|_{L^{p}(M)} + B \| f \|_{L^{p}(M)}.
\]

A possible way\footnote{at least from the PDE viewpoint, but it is not the only classical approach. Another way, with a functional analytic slant, relies on {\it Riesz transform} techniques; \cite{Ste-Singular}. We shall adopt this viewpoint in the context of Riemannian manifolds.} to attack this problem consists in mimicking what is typically done in the Euclidean space. Namely, we take a solution $u$ represented via the Green kernel of the Laplace operator and we estimate this kernel and its derivatives on the base of the geometric assumptions on the underlying space. This approach, that enables one to deduce also the existence of a solution, works pretty well when $(M,g)$ possesses a minimal, positive Green function\footnote{in the potential theoretic terminology, $M$ is called {\it non-parabolic}.} and the (Ricci) curvature is nonnegative. In this setting one even gets precise pointwise estimates that imply an $L^{p}$-control on the growth over increasing balls. We refer the reader to the fundamental work \cite{NST-JDG} by L. Ni, Y, Shi and L.F. Tam. In hyperbolic-like situations, i.e., when the (Ricci) curvature is lower bounded and there is a spectral gap, (existence and) a growth control on the solution is investigated by L. Ni, \cite{Ni-Indiana},  and in the recent papers \cite{MSW-AdvMath} by O. Munteanu, C.-J. Sung and J. Wang, and \cite{CMP-CalcVar} by G. Catino, D. Monticelli and F. Punzo. Further studies under the validity of the so called {\it weighted Poincar\'e inequalities}\footnote{where it is assumed the nonnegativity of the spectrum of a Schr\"odinger operator whose potential term is the Ricci lower bound.} can be found in the preprints \cite{MSW-preprint, CMP-preprint}. It is interesting to point out that, however, none of these nice papers investigates, a-priori, the global $W^{2,p}$ class of the solutions of the Poisson equation. \smallskip

The aim of the present paper is to give a survey of some recent results and techniques that, avoiding any use of the Green kernel of the space,  enable one to deduce the global validity of a-priopri $L^{p}$ Hessian and gradient estimates for solutions of the Poisson equation on a Riemannian manifold.
\smallskip

The survey will move from the (quite well understood) Hilbertian case to the $L^{p}$ setting, showing via counterexamples the extent to which the geometric assumptions are needed, and touching the case of manifold valued maps, where the Poisson equation takes the form of the prescription of the {\it  tension field} along a given map.\smallskip

The exposition is clearly sensitive of my personal taste and  based on my own contributions to this kind of global Calder\'on-Zygmund estimates on manifolds. I hope that  related (even fundamental) results in the literature have not been missed. On the other hand, there are topics that will not be covered, although they are interesting and very close to the subject of this survey. Among them, I would like to mention the so called {\it disturbed} or {\it weighted} Calder\'on-Zygmund inequalities where the Riemannian measure is perturbed by a suitable weight that encodes the (unbounded) geometry of the space. The interested reader is referred to the papers \cite{Am-MathZ} by E. Amar and \cite{IRV-IMRN} by D. Impera, M. Rimoldi and G. Veronelli.

\section*{Acknowledgments}

This paper is modelled on a talk given on February 2020 in the {\it S\'eminaire Th\'eorie Spectrale et G\'eom\'etrie} at the Institute Fourier, Grenoble. It's my pleasure to thank G\'erard Besson, Baptiste Devyver, Luca Rizzi and Andrea Seppi for the invitation, for the amazing hospitality and for interesting discussions around the topics of the talk. I'm also grateful to Gilles Carron, Carlo Mantegazza, Stefano Meda and Giona Veronelli for their interest in this work, for enlightening suggestions and for several corrections over the preliminary version of the paper.

\section{Basic notation}
{\bf Manifolds.} In what follows, unless otherwise specified, we always assume that $(M,g)$ is a smooth, connected {\it Riemannian manifold} of dimension $\dim M = m \geq 2$ and without boundary $\partial M = \emptyset$. In local coordinates $(x^{1},\cdots,x^{n})$, the metric coefficients are denoted by $g_{ij} = g(\partial_{i}, \partial_{j})$ where $\partial_{i} = \partial / \partial x^{i}$. Very often, to simplify the notation, we write $|X|^{2}$ instead of $g(X,X)$.\smallskip

{\bf Curvatures}. The {\it Riemann curvature tensor} of $(M,g)$ is denoted by
\[
\riem(X,Y,Z,W) = g(D_{X}D_{Y}Z - D_{Y}D_{X}Z - D_{[X,Y]}Z, W)
\]
where $D$ is the {\it Levi-Civita connection}. This latter gives rise to the {\it sectional curvature} of $M$ along the $2$-pane spanned by the linearly independent tangent vectors $X$ and $Y$ by the formula
\[
\sec(X\wedge Y) = \frac{\riem(X,Y,Y,X)}{g(X,X)g(Y,Y) - g(X,Y)^{2}}.
\]
By saying that $M$ has curvature bounded from below (resp. from above) by $K$ we mean that, for every $x \in X$ e for every $2$-plane $\Pi_{x}$ tangent to $M$ at $x$ it holds $\sect(\Pi_{x}) \geq K$ (resp. $\leq K$).

If we trace the Riemann curvature tensor we get the {\it Ricci curvature}
\[
\ric(X,X) = \sum_{j=1}^{n}\riem(X,E_{j},E_{j},X)
\]
where $\{ E_{j}\}$ is a local o.n. frame field. Equivalently, if $|X| = 1$ and $E_{1},\cdots,E_{m-1}\in X^{\perp}$, then
\[
\ric(X,X) = \sum_{j=1}^{n} \sect(X\wedge E_{j}).
\]
Inequalities involving the (metric tensor $g$ and the) Ricci tensors $\ric$ are always understood in the quadratic form sense. Namely, $\ric$ is lower bounded (resp. upper bounded) by $K$ whenever $\ric(X,X) \geq K |X|^{2}$ (resp. $\leq K |X|^{2}$) for every $x \in X$ and for every vector $X$ tangent to $M$ at $x$. In this case, we simply write $\ric \geq K$.

Finally, by tracing the Ricci tensor one gets the {\it scalar curvature} function
\[
\scal = \sum_{j=1}^{m} \ric (E_{j},E_{j}),
\]
where, again, $\{ E_{j}\}$ is any local o.n. frame field.\smallskip

{\bf Operators}. The {\it gradient}, the {\it Hessian} and the {\it Laplace-Beltrami operator} of a  smooth function $u$ are denoted, respectively, by $\nabla u$, $\Hess(u) = Dd u$ and
\[
\Delta u = \tr_{g} \Hess(u) = \dive \nabla u.
\]
Here,  the divergence is defined with the sign convention $\div X = \sum_{j=1}^{m} g(D_{E_{i}}X,E_{i})$ for a local o.n. frame field $\{ E_{j}\}$. In particular, the Laplace-Beltrami operator is nonpositive.\smallskip

{\bf Metric objects}. The symbol $d(x,y)$ stands for the {\it intrinsic (length) distance} of $M$ between $x$ and $y$. The corresponding open ball centered at $x$ and of radius $R>0$ is $B_{R}(x)$.\smallskip

{\bf Measure objects}. Unless otherwise is specified, integrals are always computed with respect to the {\it Riemannian measure} $\dv = \sqrt{\det[g_{ij}]}dx^{1}\cdots dx^{m}$. To simplify the notation we often omit to specify the measure and simply write $\int_{M} f$. The $L^{p}$-norm of a function (or a tensor) is denoted with $\| \cdot \|_{L^{p}}$ being understood that it is referred to the whole space $M$. In case it is needed to restrict the attention to functions or tensors defined on a domain $\O$ of $M$ we write $\| \cdot \|_{L^{p}(\O)}$

\section{The $L^{2}$ setting: preparatory discussions}

\subsection{$L^{2}$ gradient and Hessian estimates in $C^{\infty}_{c}$}
To start with, let us assume that $(M,g)$ is a compact Riemannian manifold of dimension $\dim M = m\geq 2$. We consider a solution $u \in C^{\infty}(M)$ of the {\it Poisson equation}
\[
\Delta u = f(x),\quad \text{on } M,
\]
for some given datum $f \in C^{\infty}(M)$. Integrating by parts this equation and using  Young inequality we obtain, for every $\e >0$,
\begin{align*}
 \int_{M} |\nabla u|^{2} &= - \int_{M}f\, u\\
 &\leq \frac{\e^{2}}{2} \int_{M}u^{2} + \frac{1}{2\e^{2}}\int_{M} f^{2} 
\end{align*}
that is, the following {\it $L^{2}$-gradient estimate} holds:
\begin{equation}\label{L2gradient}
\| \nabla u \|_{L^{2}} ^{2} \leq \| f\, u \|_{L^{1}} \leq \frac{\e^{2}}{2} \| u \|_{L^{2}} ^{2 } + \frac{1}{2\e^{2}} \|f \|_{L^{2}} ^{2}.
\end{equation}
On the other hand, since $M$ is compact its Ricci tensor is lower bounded and we can assume that $\ric \geq -K^{2}$, with $K \geq 0$. Integrating by parts the Bochner identity
 \[
 \frac{1}{2}\Delta |\nabla u|^{2} = |\Hess(u)|^{2}+  g(\nabla \Delta u ,\nabla u) + \ric(\nabla u,\nabla u)
 \]
we therefore get
\begin{align*}
0 &= \frac{1}{2}\int_{M} \Delta |\nabla u|^{2}\\
&= \int_{M}|\Hess(u)|^{2} + \int_{M} g ( \nabla \Delta u, \nabla u ) + \int_{M} \ric(\nabla u, \nabla u)\\
&\geq \int_{M}|\Hess(u)|^{2} - \int_{M} (\Delta u)^{2} - K^{2}\int_{M} |\nabla u|^{2}\\
&= \int_{M}|\Hess(u)|^{2} - \int_{M} (\Delta u)^{2} +K^{2} \int_{M} u \Delta u.
\end{align*}
Whence, using again Young inequality in the last integral we conclude the validity of the  {\it ${L^{2}}$-Hessian estimate}:
\begin{equation}\label{L2Hessian}
\| \Hess(u) \|_{L^{2}}^{2}\leq
\frac{K^{2}\e^{2}}{2}\| u \| _{L^{2}}^{2} +\left(1+\frac{K^{2}}{2\e^{2}} \right)  \| f \|_{L^{2}}^{2}.
\end{equation}
Observe that, in order to derive \eqref{L2gradient} and \eqref{L2Hessian}, the compactness of $M$ is used to perform integration by parts and to insure that the Ricci tensor is lower bounded. It follows that everything survives without changes in the non-compact setting up to using compactly supported functions and to assume the validity of the curvature condition. Summarizing, we have  obtained the following global estimate that, in elliptic regularity theory, is known as the {\it $L^{2}$ Calder\'on-Zygmund inequality}.
\begin{proposition}\label{prop-L2CZ-cs}
 Let $(M,g)$ be a Riemannian manifold of dimension $m \geq 2$ and satisfying $\ric \geq - K^{2}$ for some constant $K \geq 0$. Then, for any given $\e >0$, there exists a constant $C = C(K,\e)>0$ such that the following estimate
 \begin{equation}\label{L2CZ-cs}
 \| u \|_{W^{2,2}}:= \| u \|_{L^{2}} + \| \nabla u \|_{L^{2}} + \| \Hess(u)\|_{L^{2}} \leq C \left \{ \| u \|_{L^{2}} + \| \Delta u \|_{L^{2}} \right\}
 \end{equation}
holds for every function $u \in C^{\infty}_{c}(M).$
\end{proposition}

\subsection{Isometric immersions}\label{section-isometric-euclidean}
In the context of isometric immersions into the Euclidean space,  $L^{p}$-Hessian estimates of the form \eqref{L2Hessian}, when applied to the component functions of the map, have a clear geometric interpretation and display some interesting applications to the pre-compactness theory, as we are going to outline.\smallskip

Let $ \bPsi =(\Psi^{1},\cdots,\Psi^{n}): M^{m} \to \rr^{n}$ be a smooth immersion of the compact, $m$-dimensional manifold $M$ into $\rr^{n}$. The standard flat metric $g_{E}$ of the Euclidean space is pulled back to $M$ via $\bPsi$ and gives rise to the {\it $1^{\text{st}}$ fundamental form} of the immersion
\begin{equation}
g = \bPsi^{\ast} g_{E}. 
\end{equation}
In particular, we have that
\begin{equation}\label{isom-lip}
| d\bPsi|^{2} = \sum_{i=1}^{n} |\nabla \Psi^{i}|^{2} \equiv 1
\end{equation}
that is, $\bPsi$ is a $1$-Lipschitz map. The extrinsic geometry of the submanifold is governed by its second fundamental form which is defined as the family of vector valued bilinear forms $\bA_{x} : T_{x}M \times T_{x}M \to \rr^{n}$ such that
\begin{equation}
\bA_{x}(v,w) = ( \Hess (\Psi^{1})(x)(v,w),\cdots, \Hess(\Psi^{n})(x)(v,w)).
\end{equation}
Up to interpreting the Hessian of a vector valued map as the vector filed of the Hessian of its components, the previous formula can be written in the form
\begin{equation}\label{isom-hessian}
 \bA = {\mathrm{\mathbf{Hess}}}(\bPsi).
\end{equation}
Observe that by differentiating the identity \eqref{isom-lip} it follows that, for every $v\in T_{x}M$, $\bA_{x} (v,v) \in T_{x}M^{\perp}= N_{x}M \subseteq \rr^{n}$ the normal space to $M$ at $x$. Obviously, here we are identifying $T_{x}M \approx d_{x}\bPsi(T_{x}M) \subseteq \rr^{n}$. Taking the trace of the second fundamental form gives the normal vector field $\bH(x) \in N_{x}M$, called the {\it mean curvature vector field of the immersion}, and from \eqref{isom-hessian} we get
\begin{equation}\label{isom-lap}
 \bH(x) = \mathbf{\Delta} \bPsi
\end{equation}
where, as above, the Laplacian of a vector valued map is nothing but the vector of the Laplacians of its components. The immersion $\bPsi:M \to \rr^{n}$ is called minimal if $\bH\equiv 0$ and totally geodesic if $\bA \equiv 0$. In this latter case, $\bPsi(M)$ is inside an affine $m$-plane of $\rr^{n}$.

With this notation, if we have an $L^{2}$-Hessian estimate like \eqref{L2Hessian} and we apply it to each component of $\bPsi$ we get
\begin{equation}\label{sff}
\int_{M} |{\bf A} |^{2} \leq \mathrm{CZ}_{2} \left\{ \int_{M} \dist_{\rr^{n}}(\bPsi , {\bf 0})^{2} + \int_{M} |{\bf H}|^{2 }\right\},
\end{equation} 
where $\mathrm{CZ}_{2}$ depends suitably on the geometry of $M$. Namely, the $L^{2}$-norm of the second fundamental form is controlled by the $L^{2}$-norm of the mean curvature via an extrinsic $L^{2}$-diameter bound. In Section \ref{section-maps} we will extend this kind of estimate to general ambient manifolds and general exponents $p$ and, moreover,  the dependence of $\mathrm{CZ}_{p}$ on the $C^{1,\a}$ geometry will be emphasized.

\begin{remark}
Note that, obviously, the extrinsic diameter of the source compact manifold:
\[
\diam_{\rr^{n}}(M) = \sup\{ \dist_{\rr^{n}} (\bPsi(x) , \bPsi(y) ) : x,y \in M \}
\]
is dominated by its intrinsic diameter:
\[
\diam(M) = \sup\{d(x,y): x,y \in M \}.
\]
We point out that this latter, in turn, can be estimated in terms of the volume of $M$ and its {\it Euclidean radius},  namely, the radius (independent on the center) of the largest ball inside a local coordinate chart where $g$ and $g_{E}$ satisfy $2^{-1} g_{E} \leq g \leq 2 g_{E}$ in the sense of quadratic forms. We shall come back to this in  Section \ref{section-radii} below where we will collect some of the main definitions of Riemannian radii and their basic estimates. Here, we limit ourselves to observe that, given $V,E>0$, if
 \begin{equation}\label{alternative}
 i)\, \vol M \leq V,\quad ii)\, r_{\euc}(x) \geq E,\, \forall x\in M,
\end{equation}
then there exists a constant $D=D(V,E)>0$ such that
\[
\diam(M) \leq D.
\]
Simply choose $x,y \in M$ such that $\diam M = d(x,y)$, take a minimizing geodesic $\g:[0,1] \to M$ connecting $\g(0) =x$ to $\g(1) = y$ and put over $\g([0,1])$  a number $h = \lfloor \diam M / 4E \rfloor \in \nn$ of disjoint balls $B_{E/2}(x_{i})$ centered at $x_{j} \in \g([0,1])$. Since, by \eqref{alternative} ii), $\vol B_{E/2}(x_{j}) \geq 2^{-m/2} \vol_{\euc} \bb_{E}(0)=:v$ and the balls are disjoint, we get $h \cdot v \leq V$. Whence the announced diameter estimate follows.
\end{remark}

When combined with elliptic regularity theory in the geometric setting of convergence of Riemannian manifolds,  this enables one to obtain pre-compactness conclusions in the following class
{\small \[
 \M(m,n,V,E,H,\mathrm{CZ}_{2}) = \{ \bPsi:(M^{m},g) \to \rr^{n}: M \text{ cmpt }, \vol(M) \leq V, r_{\euc}\geq E,|\bH|\leq H, \eqref{sff}\text{ holds}\}.
\]
}More precisely, we obtain the following corollary of a by now classical result for surfaces due to  J. Langer, \cite {La}, recently extended to any dimension and codimension by  P. Breuning, \cite{Br}. We do not insist in introducing all the definitions related to the geometric convergence theory  but instead we leave the statement in a rather ``suggestive form'' and strongly recommend the interested reader to consult the very well written and informative paper \cite{Sm} by G. Smith.
\begin{proposition}
 Let us given a sequence of immersions
 \[
 \{ \bPsi_{k}:M_{k} \to \rr^{n}\} \subseteq \M(m,n,V,E,H,\mathrm{CZ}_{2}).
 \]
 Then, there exists a $C^{1,\a}$ compact Riemannian manifold $(M,g)$ of $\dim M =m$ and a $C^{1,\a}$ isometric immersion $\bPsi : M \to \rr^{n}$ such that, for a suitable subsequence $\{ k_{j} \}$ and a sequence of points $\by_{j} \in \rr^{n}$ it holds that $M_{k_{j}} \to M$ and $\bPsi_{k_{j}} - \by_{j} \to \bPsi$ in the (Cheeger-Gromov) $C^{1,\a}$-topology.
\end{proposition}

\subsection{From $C^{\infty}_{c}$ to $L^{2}$: Sobolev spaces and cut-off functions}

 A natural question is the following:\smallskip
\begin{problem}
 What happens if $(M,g)$ is a complete, noncompact Riemannian manifold satisfying $\ric \geq -K^{2}$ and the class of functions $C^{\infty}_{c}(M)$ is replaced by  $L^{2}(M)$? Do the $L^{2}$ gradient and Hessian estimates survive?
\end{problem}

The answer is yes but switching from $C^{\infty}_{c}$ to $L^{2}$ is not completely trivial as it could appear. The crucial ingredients will be represented by special cut-off functions with controlled derivatives. We are going to split the discussion in two parts.\smallskip

\subsubsection{$L^{2}$-gradient estimate: $1^{st}$ order cut-off functions}

These estimates are naturally related to the spectral properties of the Laplace-Beltrami operator. \smallskip

Recall that a symmetric, densely defined, unbounded operator $T : \D\subseteq \CH \to \CH$ on a Hilbert space $\CH$ is called {\it essentially self-adjoint} if its closure $\bar T$ (in the graph sense) is a self-adjoint operator. It was first observed by R. Strichartz in \cite{Str-JFA}, that on a geodesically complete Riemannian manifold $(M,g)$ the Laplace operator $-\Delta : C^{\infty}_{c}(M) \subseteq L^{2}(M) \to L^{2}(M)$ is essentially self-adjoint. In fact, according to \cite[pp. 136, 137]{ReSi}, it is enough to show that for any $\l >0$, the equation
\[
\Delta u = \l u \quad \text{on } M
\]
has no non-zero $L^{2}$-solutions. One can verify that this is the case throughout a Caccioppoli-type inequality of the form
\[
\int_{M} |\nabla v|^{2} \vr_{k}^{2} \leq C \int_{M} v^{2} |\nabla  \vr_{k}|^{2}
\]
where $C>0$ is an absolute constant, $v = u_{\pm}$ and $\{ \vr_{k}\} \subseteq C^{\infty}_{c}(M)$ is a sequence of cut-off functions satisfying
\begin{equation}\label{cutoff1}
a)\, 0 \leq \vr_{k} \leq 1,\quad b)\, \vr_{k} \nearrow 1,\quad c)\, \|\nabla \vr_{k} \|_{L^{\infty}} \searrow 0,
\end{equation}
as $k \to \infty$. The existence of this special sequence of {\it $1^{st}$ order cut-off functions} (in the language of \cite{Gu-JGEA}) is in fact a characterization of the geodesic completeness of $(M,g)$; see e.g. \cite{PS-Ensaios} for an $\infty$-parabolic viewpoint.

Thus, if we define
\[
\D(-\Delta_{max}) = \{ u \in L^{2}(M): \text{the distribution }\Delta u \in L^{2}(M)\}
\]
and
\[
\D(-\Delta_{min}) = \{ u \in L^{2}(M):  u = \lim_{L^{2}}\vp_{j}, \,\exists  \lim_{L^{2}}\Delta \vp_{j}=:\Delta u,\, \text{ for some } \{ \vp_{j}\}\subset C^{\infty}_{c}(M)\}
\]
the essential self-adjointness of $-\Delta$ implies that
\begin{equation}\label{essself}
\D(-\Delta_{max}) = \D(-\Delta_{min}).
\end{equation}
Now, let $u, \Delta u \in L^{2}(M)$. Since
\[
 u = \lim_{L^{2}}\vp_{j},\quad \Delta u = \lim_{L^{2}}\Delta \vp_{j},
\]
with $\vp_{j}\in C^{\infty}_{c}(M)$, by applying to $\vp_{j}$ the gradient estimate \eqref{L2gradient} we get
 \begin{equation}\label{vf}
 \| \nabla \vp_{j}  \|^{2}_{2} \leq \frac{\e^{2}}{2} \| \vp_{j} \|_{L^{2}}^{2}  + \frac{1}{2\e^{2}} \| \Delta \vp_{j} \|_{L^{2}}^{2}.
 \end{equation}
This latter yields $\nabla u = \lim_{L^{2}} \nabla \vp_{j}$. Indeed, clearly, the sequence of $L^{2}$ vector fields $\{ \nabla \vp_{j}\}$ is Cauchy and, therefore, it converges in $L^{2}$ to some vector field $X$. A weak convergence argument now shows that $X = \nabla u$ as claimed. It follows that we can take the limit as $j \to +\infty$ into \eqref{vf} and deduce the validity of the following
\begin{proposition}\label{prop-L2gradient}
 Let $(M,g)$ be a complete Riemannian manifold and let $u \in C^{\infty}(M)$ be a solution of the Poisson equation $\Delta u = f$ on $M$. If $u,f \in L^{2}(M)$ then $|\nabla u| \in L^{2}(M)$ and, for any $\e>0$,
 \[
 \| \nabla u  \|^{2}_{2} \leq \frac{\e^{2}}{2} \| u \|_{L^{2}}^{2}  + \frac{1}{2\e^{2}} \| f \|_{L^{2}}^{2}.
\]
\end{proposition}

\subsubsection{Sobolev spaces and density problems}\label{section-density}
In view of forthcoming discussions, and of the fact that Sobolev space theory is the red wire connecting the various parts of the present survey, we reformulate condition \eqref{essself} and its subsequent implication in a slightly different way. To this end, we need to introduce some notation.  Given $1<p<+\infty$, and $u \in L^{p}(M)$ denote by $\nabla u$, $\Delta u$ and $\Hess(u)$, respectively, the distributional gradient, Laplacian and Hessian of $u$. Then, define the Sobolev space
\[
 W^{1,p}(M) = \{ u \in L^{p}(M): |\nabla u| \in L^{p}(M)\}
\]
with norm
\[
\| u \|_{W^{1,p}} = \| u \|_{L^{p}} + \| \nabla u \|_{L^{p}};
\]
the Sobolev space
\[
 W^{2,p}(M)= \{ u \in L^{p}(M) : |\nabla u|,\, |\Hess(u)| \in L^{p}(M)\}
\]
with norm
\[
 \| u \|_{ W^{2,p} } = \| \Hess( u ) \|_{L^{p}} + \| \nabla u \|_{L^{p}} + \| u \|_{L^{p}}
\]
and, finally, the Sobolev space
\[
\tW^{2,p}(M) = \{ u \in L^{p}(M): \Delta u \in L^{p}(M)\}
\]
with norm
\[
\| u \|_{\widetilde W^{2,p} } = \| \Delta u \|_{L^{p}} + \| u \|_{L^{p}}.
\]
All of them are reflexive Banach spaces. Moreover, if we limit ourselves to Sobolev functions that ``vanish at infinity'' (in the sense we are going to define) we give rise to new Banach spaces inside $W^{1,p}$, $\widetilde W^{2,p}$ and $W^{2,p}$. They are defined, respectively, as the closures
\[
W^{1,p}_{0}(M) = \overline{C^{\infty}_{c}(M)}^{W^{1,p}} \subseteq W^{1,p}(M)
\]
and
\[
\tW^{2,p}_{0}(M) = \overline{C^{\infty}_{c}(M)}^{\tW^{2,p}} \subseteq \tW^{2,p}(M),\quad W^{2,p}_{0}(M) = \overline{C^{\infty}_{c}(M)}^{W^{2,p}}\subseteq W^{2,p}(M).
\]
The {\it density problem in Sobolev space theory} consists in understanding geometric conditions under which these inclusions are, in fact, equalities. The importance of this basic problem relies on the fact that it is very often much more  convenient to work with smooth compactly supported functions and, therefore, one has to be sure that these are enough to approximate every Sobolev function. In this respect, it is important to stress that, on a complete manifold, there is no loss of information in assuming that Sobolev functions are in fact smooth thanks to the following (special case of the) Meyers-Serrin type result; see \cite{GGP}.
\begin{theorem}\label{th-MS}
 Let $(M,g)$ be a complete Riemannian manifold. Then
 \[
 \overline{W^{k,p}(M)\cap C^{\infty}(M)}^{W^{k,p}} = W^{k,p}(M)
 \]
 for every $1 \leq p<+\infty$ and $k=1,2$ (actually for every $k \in \nn$ up to defining the higher order Sobolev spaces).
\end{theorem}

In this framework, what Strichartz observed can be formulated by stating that, thanks to the existence of $1^{st}$ order cut-off function \eqref{cutoff1}, the following density result holds.
\begin{proposition}[Strichartz]\label{prop-Strichartz}
 Let $(M,g)$ be a complete Riemannian manifold. Then
 \[
 W^{1,2}(M) = W^{1,2}_{0}(M)
 \]
 and
 \[
\tW^{2,2}(M) = \tW_{0}^{2,2}(M).
 \]
\end{proposition}
Later on in these notes we will touch the problem of extending the second density conclusion to integrability exponents $p \not=2$ (the extension of the first one is almost trivial). In fact we shall see that, quite surprisingly, geodesic completeness is enough to include all the $L^{p}$ scale, $1<p<+\infty$; see Theorem \ref{th-Milatovich}.
\subsubsection{$L^{2}$-Hessian estimates: Laplacian cut-off functions} The extension to the genuine $L^{2}$ setting of \eqref{L2Hessian} was first observed by L. Bandara in \cite{Ba-PAMS}. The discussion presented here incorporates contributions and viewpoints from \cite{Gu-JGEA, Gu-book} by B. G\"uneysu and from the very recent \cite{IRV-IMRN} by  D. Impera, M. Rimoldi and G.Veronelli.\smallskip

Everything here boils down to the next nontrivial result of independent interest.
\begin{theorem}\label{th-L2Sobolev}
 Let $(M,g)$ be a complete Riemannian manifold satisfying $\ric \geq -K^{2}$ for some $K \geq 0$. Then, the following chain of equalities holds:
 \begin{equation}
 \widetilde W^{2,2}(M) = \tW^{2,2}_{0}(M) = W^{2,2}_{0}(M) = W^{2,2}(M).
 \end{equation}
\end{theorem}
Before coming into the proof of Theorem \ref{th-L2Sobolev}, let us see how it implies immediately the $L^{2}$-Hessian estimate.
\begin{corollary}\label{cor-L2Hessian}
  Let $(M,g)$ be a complete, $m$-dimensional Riemannian manifold satisfying $\ric \geq -K^{2}$ for some $K \geq 0$. Then, for every solution $u \in C^{\infty}(M)\cap L^{2}(M)$ of the Poisson equation $\Delta u = f$ with $f \in L^{2}(M)$, the Hessian estimate \eqref{L2Hessian} holds true.
\end{corollary}
\begin{proof}
 Indeed, since $u \in \widetilde{W}^{2,p}(M)$, then $u \in W^{2,p}(M)$ and by density we find a sequence $\{ \vp_{j}\} \subset C^{\infty}_{c}(M)$ such that
\[
i)\, u = \lim_{L^{2}} \vp_{j},\quad ii)\, \nabla u = \lim_{L^{2}} \nabla \vp_{j},\quad iii) \Hess(u) = \lim_{L^{2}} \Hess(\vp_{j}).
\]
Applying \eqref{L2Hessian} to each $\vp_{j}$ we get
\[
\|  \Hess (\vp_{j}) \|_{L^{2}} \leq C_{1} \| \vp_{j}\|_{L^{2}} + C_{2} \|  \Delta \vp_{j} \|_{L^{2}}
\]
and the conclusion follows by taking the limit as $j \to \infty$.
\end{proof}

\begin{proof}[Outline of the proof of Theorem \ref{th-L2Sobolev}]
According to Proposition \ref{prop-Strichartz} we shall focus on the equalities
\begin{itemize}
 \item [(a)] $\tW^{2,2}_{0}(M) = W^{2,2}_{0}(M)$
 \item [(b)] $W^{2,2}_{0}(M) = W^{2,2}(M)$.
\end{itemize}
 Equality (a) is the easier of the two. Indeed, by the Cauchy-Schwarz inequality, for every $\vp \in C_{c}^{\infty}(M)$,
\[
| \Delta \vp | \leq \sqrt{m} \, |\Hess(\vp)|
\]
proving that
\[
\| \vp \|_{\tW^{2,2}} \leq C_{1} \| \vp \|_{W^{2,2}}
\]
for some $C_{1}=C_{1}(m)>0$ and, therefore,
\[
W^{2,2}_{0}(M) \subseteq \tW^{2,2}_{0}(M).
\]
On the other hand, by applying the $L^{2}$-Hessian estimate \eqref{L2Hessian} for compactly supported functions we see that
\[
\| \vp \|_{W^{2,2}} \leq C_{2} \| \vp \|_{\tW^{2,2}}
\]
for some constant $C_{2}=C_{2}(K)>0$ and, therefore, we also have the opposite inclusion
\[
\tW^{2,2}_{0}(M) \subseteq W^{2,2}_{0}(M).
\]
This finishes the proof of (a).

The proof of (b) relies on the following important lemma by R. Schoen and S.T. Yau, \cite{SY-redbook}.
\begin{lemma}[Laplacian cut-off functions]
 Let $(M,g)$ be a complete Riemannian manifold of dimension $m$ and Ricci curvature $\ric \geq -K^{2}$. Let also $o \in M$ be a fixed reference point. Then, there exists a sequence of cut-off functions $\vp_{j} \in C^{\infty}_{c}(M)$ satisfying
 \begin{equation}\label{cutoff2}
 a)\, 0 \leq \vp_{j} \leq 1,\quad b)\,  \vp_{j} \nearrow 1, \quad c)\, \|\nabla \vp_{j}\|_{L^{\infty}} \to 0,\quad d)\, \| \Delta \vp_{j} \|_{L^{\infty}} \leq C,
\end{equation}
for some constant $C=C(m,K, \operatorname{Geom}(B_{1}(o)))>0$ and for $j \to +\infty$.
\end{lemma}
In the language of \cite{Gu-JGEA, IRV-IMRN} these $2^{nd}$ order cut-off functions are called {\it (weak\footnote{meaning that in d) we {\bf do not} require that $\| \Delta \vp_{j} \|_{L^{\infty}} \to 0$ as $j \to+\infty$.}) Laplacian cut-off functions}.

Now, given $u \in C^{\infty}(M) \cap W^{2,2}(M)$, one is led to define $u_{j} = \vp_{j}u \in C^{\infty}_{c}(M)$ and try to show that $\| u_{j} - u\|_{W^{2,2}} \to 0$ as $j \to +\infty$. Up to the first order, there is no problem because

\[
\int_{M}(u-u_{j})^{2} = \int_{M}(1-\vp_{j})^{2}u^{2} \to 0
\]
and 
\[
\int_{M}|\nabla u - \nabla u_{j}|^{2}\leq 2 \int_{M}|\nabla u|^{2}(1-\vp_{j})^{2} + 2 \int_{M}u^{2}|\nabla \vp_{j}|^{2} \to 0
\]
as $j \to +\infty$. If we try the same at the second order we realize that we have only a control on $\Delta \vp_{j}$ not on on the full $\Hess(\vp_{j})$. However we can still invoke the help of the Bochner formula
 \[
 \frac{1}{2}\Delta |\nabla \vp_{j}|^{2} = |\Hess(\vp_{j})|^{2}+  g(\nabla \Delta \vp_{j} ,\nabla \vp_{j}) + \ric(\nabla \vp_{j},\nabla \vp_{j})
 \]
 and use suitable integration by parts to replace $|\Hess(\vp_{j})|$ with $|\Delta \vp_{j}|$ in the estimate of $\int_{M} | \Hess(u) - \Hess(u_{j})|^{2}$. Joint with the properties of $\vp_{j}$ this yields that
\[
\int_{M} | \Hess(u) - \Hess(u_{j})|^{2} \to 0,
\] 
as $j \to +\infty$.
\end{proof}
\begin{remark}
The proof of (b) we have proposed is based on the existence of special cut-offs. There is (at least) a second natural way to obtain the equality $W^{2,2}_{0}(M) = W^{2,2}(M)$ in the above assumptions. It uses only the $L^{2}$-Hessian estimate on compactly supported functions. See Proposition \ref{prop-density} below. The reason why we have emphasized the cut-off function approach is that it is much more powerful and flexible in terms of geometric conditions, \cite{IRV-IMRN}. Summarizing: the validity of Calder\'on-Zygmund is a (strictly) stronger condition than the density property in Sobolev spaces. In this sense the relation between Calder\'on-Zygmund and density must be used in the other direction, i.e., in terms of counterexamples. Section \ref{CZ-counterxamples} should clarify this sentence.
\end{remark}

\subsubsection{Final considerations on the $L^{2}$-setting}
Let us emphasize some of the main ingredients and interplays emerging from the previous discussions. Let $(M,g)$ be a noncompact Riemannian manifold.\medskip

\noindent (A) Integration by parts shows that, for any $u \in C^{\infty}_{c}(M)$, $\| \nabla u \|_{L^{2}} \leq \| u \Delta u \|_{L^{1}}$. If $M$ is geodesically complete, this inequality extends to $u \in C^{\infty}(M)$ satisfying $u,\Delta u \in L^{2}(M)$. This is related to the essential self-adjointeness  of $\Delta$ and can be formulated in terms of a density property of compactly supported smooth functions in $\tW^{2,2}(M)$.\medskip

\noindent (B) Assuming that $\ric \geq -K^{2}$, integrating by parts the Bochner formula gives that, for every $u \in C^{\infty}_{c}(M)$, $\| \Hess (u) \|_{L^{2}} \leq C_{1} \| u \|_{L^{2}} + C_{2} \| \Delta u \|_{L^{2}}$. If further $M$ is geodesically complete, this inequality extends to $u \in C^{\infty}(M)$ with $u,\Delta u \in L^{2}(M)$. This extension relies on the equality of Sobolev spaces established in Theorem \ref{th-L2Sobolev}. In particular, it needs the equivalence of the Sobolev norms of $\tW^{2,2}$ and $W^{2,2}(M)$, and the density of $C^{\infty}_{c}(M)$ in the Sobolev space $W^{2,2}(M)$. In fact, this latter property holds under a quadratic Ricci decay $\ric \geq - C(1+r(x)^{2})$, see \cite{IRV-IMRN}.
\medskip

\noindent (C) Density properties in Sobolev spaces $W^{k,p}$ follow from the possibility of producing sequences of cut-off functions with controlled derivatives up the order $k$. As we shall see later, a natural way to construct these cut-off functions is to use {\it distance like} functions which are smooth and with controlled derivatives up to the order $k$. This is what was done in the seminal work of Schoen-Yau quoted before. Subsequently, the Schoen-Yau construction was extended to quadratic Ricci decays by D. Bianchi and A.G. Setti, \cite{BS-CalcVar} but in their work only a control on the Laplacian is introduced. The genuine control on the higher derivatives of distance like functions in unbounded (but controlled) geometries is a  contribution of Impera-Rimoldi-Veronelli, \cite{IRV-IMRN, IRV-preprint}. We refer the reader to these papers, actually based on \cite{BS-CalcVar}, for the up-to-date conditions joint with further very nice applications.\medskip

Density questions in Sobolev spaces are naturally related to curvature. This sentence is strongly true in the sense that, on the one side, controlling the curvature implies density as discussed above and, on the other side, a very recent and striking counterexample due to Veronelli, \cite{Ve-Counterexample},  shows that density needs a curvature restriction.

\begin{problem}
 Is it similar for the $L^{2}$ Calder\'on-Zygmund inequality at the $C^{\infty}_{c}(M)$ level? Namely, we have seen that a Ricci lower bound implies their validity. What happens if the curvature is unbounded?
\end{problem}

 \begin{example}\label{example-L2}
The first example showing that, in general, the $L^{2}$ Hessian estimate for compactly supported functions may fail if the curvature is unbounded, was discovered in \cite{GP-AdvMath}. Briefly, the example is constructed as follows.\smallskip

Let $\rr^{2}$ be endowed with the complete Riemannian metric $g$ that, in polar coordinates $(r,\theta)$, writes as
\[
g= dr\otimes dr + \s^{2}(r)d\theta \otimes d\theta
\]
where $\s: \rr \to \rr$ is a smooth function satisfying the following requirements
\[
i)\,\, \s(t)>0\,\, \forall t>0,\quad ii)\,\, \s^{(2k)}(0) = 0 \,\, \forall k \in \nn,\quad iii) \,\, \s'(0) =1.
\]
The corresponding Riemannian manifold $(\rr^{2},g)$, usually called {\it a model manifold with warping function $\s$}, will be denoted by $\mm^{2}_{\s}$.

We put ourselves in the setting where $\mm^{2}_{\s}$ has infinite volume and its Laplace-Beltrami operator does not possess a positive Green kernel. Said in equivalent terms, $\mm^{2}_{\s}$ is an infinite-volume parabolic manifold. It is well known that this is equivalent to require that, respectively, $\int^{+\infty}\s(r) \,dr =+\infty$ and $\int^{+\infty} \frac{dr}{\s(r)} =+\infty$. More precisely, we assume that
\[
r \leq \s(r) \leq r+1
\]
for all $r \geq 1$. We are going to use the (rotationally symmetric) signed Green kernel $G : \mm^{2}_{\s}\setminus\{0\} \to \rr$ of the Laplacian, given by
\[
G(r) = \int_{1}^{r} \frac{dt}{ \sigma(t)},
\]
as a new coordinate that replaces the distance function $r$ in the polar representation of $\rr^{2}$. This looks pretty much in spirit of the {\it fake distance} constructed on $m \geq 3$-dimensional non-parabolic manifolds by T. Colding in \cite{Co-Acta} although we were not aware of this work at that time. See also the recent \cite{MRS}.

Now, having fixed $\phi \in C^{\infty}_{c}([0,1])$ and the corresponding sequence of translated functions $\phi_{k} (t)  = \phi(t-k) \in C^{\infty}_{c}( [k,k+1])$, we define the radial cut-off functions $u_{k} \in C^{\infty}_{c}(\mm^{2}_{\s})$ by
\[
u_{k}(r) = \phi_{k}\circ G(r).
\]
Then, one can verify that the following estimates hold:\smallskip
\begin{itemize}
 	\item $\|  u_{k}\|  _{2}^{2} \approx \int_{k}^{k+1}\left(\phi_{k}\left(  s\right)  \right)  ^{2}\left(  \sigma \circ  G^{-1}\left(s\right)    \right)  ^{2} ds \leq  C\, e^{2k}$\smallskip
	\item $\|  \Delta u_{k}\|  _{2}^{2} \approx {2}\int_{k}^{k+1}
\frac{\left(  \phi_{k}^{\prime\prime}\left(  s\right)  \right)  ^{2}}{\left(\sigma\circ   G^{-1}\left(  s\right)   \right)  ^{2}}ds \leq C \, e^{-2k}$\smallskip
	\item $\|  \mathrm{Hess}\left(  u_{k}\right)  \|  _{2}^{2} 
\geq \int_{k}^{k+1}\left(  \phi_{k}^{\prime}\left(  s\right)
\right)  ^{2}\left(  \dfrac{\sigma^{\prime}}{\sigma}\circ   G^{-1}\left(s\right)  \right)  ^{2} ds$.
\end{itemize}
Note that the derivatives of $\s$, in fact its first derivative, appears only in the lower estimate of the Hessian. Thus, in order to violate the $L^{2}$ Calder\'on-Zygmund inequality at the $C^{\infty}_{c}$-level we can choose $r \leq \s(r) \leq r+1$ carefully in such a way that: (a) it oscillates more and more and, (b) its first derivative $\s'(r)$ explodes sufficiently fast in the interval $[k,k+1]$ as $k \to +\infty$. The first condition is just to ensure that an ``almost vertical'' piece of $\s$ can be found in any interval $[k,k+1]$.
\end{example}

The $L^{2}$ picture is relatively well understood. Now, the basic questions that will permeate the rest of the paper.

\begin{problem}
What happens on the $L^{p}$ scale, if $1<p<+\infty$? What are the geometric conditions on $(M,g)$ ensuring that, given a smooth solution of the Poisson equation $\Delta u = f$ on $M$ with $u, f  \in L^{p}(M)$ it holds that $\| u \|_{W^{2,p}} \leq C_{1} \| u \|_{L^{p}} + C_{2} \| \Delta u \|_{L^{p}}$? How  the constants depend on the geometry? 
\end{problem}

\begin{remark}\label{rem-cz-endpoint}
 In these basic problems the endpoint cases $p=1$ and $p=+\infty$ will be not considered because the corresponding Calder\'on-Zygmund theory is false even in the Euclidean space. More precisely, concerning the Hessian estimate, we have the following counterexamples:
\begin{itemize}
 \item ({$\mathbf{p=\infty}$)} This case is intimately related with the failure of the Schauder estimates outside the H\"older setting. Indeed, the example in  \cite[Problem 4.9 (a)]{GT}, shows that there exists a function $f \in C^{0}_{c}(\rr^{n})$ such that the Poisson equation $\Delta u = f$ has a solution $u_{0} \in C^{1}(\rr^{n})\cap C^{\infty}(\rr^{n}\setminus \{0\})$ with $| \Hess(u_{0}) |(x) \to +\infty$ as $x \to 0$. In particular, $u_{0} \not \in C^{2}(\rr^{n})$ (or, better, $u \not\in C^{1,1}(\rr^{n})$). Now, if $u_{1}$ is any other locally bounded solution of the same equation, then $u_{1}-u_{0}$ is harmonic and bounded near the origin. Therefore, by elliptic regularity, it must be smooth in a neighborhood of $0$. Since $u_{1} = u_{0} + (u_{1}-u_{0})$ we conclude that also $| \Hess(u_{1})(x)| \to +\infty$ as $x \to 0$, showing that the equation $\Delta u = f$ has no $C^{2}$ (actually $C^{1,1})$ solution at all. Now, as a consequence of a very general and abstract result of K. De Leeuw and H. Mirkil, \cite[Proposition 2]{DM-CRAS}, the validity of an $L^{\infty}$-estimate like $\| \Hess(\vp) \|_{L^{\infty}} \leq C \| \Delta \vp \|_{L^{\infty}}$ for all $\vp \in C^{\infty}_{c}(\rr^{n})$ and for some universal constant $C>0$ is completely equivalent to the fact that for any function $f \in C^{0}(\rr^{n})$ satisfying $f(x) \to 0$ as $|x| \to +\infty$, it holds that $\Delta u = f$ has a solution with $\partial^{2}_{ij}u \in C^{0}(\rr^{n})$. Thus, the above example shows that no such $L^{\infty}$-estimate can be satisfied.

 \item ($\mathbf{p=1}$) In the paper \cite[Part 1]{Or-ARMA} D. Ornstein constructs a (somewhat) explicit  sequence $\vp_{k}\in C^{\infty}_{c}({\mathbf{C}}^{2}_{1}(0))$, where ${\mathbf{C}}^{2}_{1}(0)$  is the unit cube of $\rr^{2}$, for which
 \[
 \int_{\rr^{2}} | \partial^{2}_{xy} \vp_{k} | \geq k \int_{\rr^{2}} |\partial^{2}_{xx}\vp_{k}| + | \partial^{2}_{yy}\vp_{k}|.
 \] 
 In particular, the inequality $\| \Hess(\vp) \|_{L^{1}} \leq C \| \Delta \vp \|_{L^{1}}$ cannot be true on $C^{\infty}_{c}(\rr^{2})$ with a uniform constant $C>0$. In fact, in the same paper, it is proved a much more general and abstract result on the failure of the $L^{1}$ estimate, which is inspired by the $L^{\infty}$-case studied in \cite{DM-CRAS}.
\end{itemize}
\end{remark}

We are going to give an overview of what is known so far in the direction of the above problems. We shall start discussing some illuminating (counter)examples and then we will follow the scheme we have adopted in the Hilbertian case, namely, we shall investigate separately the validity of the Hessian and the gradient estimates first studying the $C^{\infty}_{c}$ case and, then, moving to the $L^{p}$ level.\smallskip

Following \cite{Gu-JGEA, GP-AdvMath}, in the sequel we shall make use of the following terminology.

\begin{definition}\label{def-CZ}
 Let $(M,g)$ be an $m$-dimensional Riemannian manifold. We say that an $L^{p}$-Calder\'on-Zygmund inequality holds on $M$ if there exists a constant $C>0$ such that
 \begin{equation}\label{CZ}\tag{CZ}
 \| \Hess(\vp) \|_{L^{p}} \leq C \left\{ \| \vp \|_{L^{p}} + \| \Delta \vp \|_{L^{p}} \right\}
\end{equation}
for every $\vp \in C^{\infty}_{c}(M)$.
\end{definition}

\section{$L^{p}$-Hessian estimates: counterexamples}\label{CZ-counterxamples}

We have seen in Example \ref{example-L2} that, in general, the validity of an $L^{2}$ Calder\'on-Zygmund inequality, \ref{CZ}(2), may fail, even at the $C^{\infty}_{c}$ level, on a $2$-dimensional, parabolic, model manifold. Since the volume of this manifold is infinite, nothing can be deduced concerning the validity or the failure of {CZ}(p) when $p \not = 2$. Moreover, the estimates we gave are simplified by the assumption that we are in a $2$-dimensional space. However, in the paper \cite{Li-AGAG}, S. Li is able to strengthen the construction so to include every dimension $m \geq 2$ and every $1<p<+\infty$, thus completing the picture.
 \begin{theorem} \label{th-Li}
 For any $m \geq 2$ and every $1<p<+\infty$ there exists a complete, parabolic, model manifold $\mm^{m}_{\sigma}:= (\rr^{m},dr \otimes dr +\sigma^{2}(r) g_{\ss^{m-1}})$ with unbounded (radial) sectional curvature  such that, along a sequence $\vp_{k} \in C^{\infty}_{c}(\mm^{m}_{\sigma})$, \ref{CZ}(p) is violated, i.e.,
 \[
\lim_{k \to +\infty} \frac{ \| \Hess(\vp_{k}) \|_{L^{p}}}{\| \vp_{k} \|_{L^{p}} + \| \Delta \vp_{k} \|_{L^{p}}} = +\infty.
 \]
\end{theorem}

Next we point out that, for large values of $p$, the constant in \ref{CZ}(p) cannot depend only on a lower sectional curvature bound. This is obtained by G. De Philippis and J. Zimbron in \cite{DZ-preprint} as a consequence of their harmonic function theory on ${\rm RCD}(n,K)$ spaces and answers in the negative a question rised in \cite{Gu-JGEA}.

\begin{theorem}\label{th-DZ}
Let $m< p <+\infty$. There exists a sequence of compact(!) Riemannian manifolds $(M^{m}_{k},g_{k})$ with $\sect_{M_{k}} \geq 0$ and $\diam ( M_{k} ) \leq D$ such that, for some $\vp_{k} \in C^{\infty}(M_{k})$,
\begin{equation}\label{DZ1}
\| \Delta_{g_{k}} \vp_{k} \|_{L^{p}} + \| \nabla_{g_{k}} \vp_{k} \|_{L^{p}} \leq 1
\end{equation}
and
\begin{equation}\label{DZ2}
\lim_{k\to +\infty} \| \Hess_{g_{k}}(\vp_{k}) \|_{L^{p}}  = +\infty.
\end{equation}
\end{theorem} 

\begin{remark}
 Note that, by volume comparison, \cite{PRS-book}, also the volumes of $M_{k}$ are uniformly upper bounded.
\end{remark}
\begin{remark}
In the above statement nothing is said about the  $L^{p}$ norm of $\vp_{k}$. However, up to translating each $\vp_{k}$ by a suitable constant $c_{k} \in \rr$ (that does not affect neither the assumptions nor the conclusion), we can assume that $\| \vp_{k}\|_{L^{p}} \leq C$ for some uniform constant $C=C(K,D,m,p)>0$. Indeed, recall from \cite[Theorem 5.6.6]{SC-book} that, on a compact $m$-dimensional Riemannian manifold $(M,g)$ with $\ric \geq -K^{2}$, $K \ge 0$, the following $L^{p}$ Neumann-Poincar\'e inequality holds
\[
\| u - \avint_{M} u \|_{L^{p}}^{p} \leq C_{1}^{p} \, \diam(M)^{p} e^{C_{2}K\diam(M)} \| \nabla u \|_{L^{p}}^{p}
\]
for every $u \in C^{\infty}(M)$ and for some constants $C_{1} = C_{1}(m,p),\, C_{2}=C_{2}(m)>0$. Thus, in the assumptions of Theorem \ref{th-DZ}, letting $c_{k} = \avint_{M} \vp_{k} $ we have that
\[
\| \vp_{k} - c_{k} \|_{L^{p}} \leq C:=C_{1} D,
\]
as claimed.
\end{remark}

\begin{remark}
Combining the above observation with the interpolation inequalities that we shall discuss in Section \ref{section-interpolation} below, we deduce the following fact of independent interest. Given any Riemannian manifold $(M,g)$ we agree to denote\footnote{the symbol $\ZM$ is the acronym of ``zero mean''. The upper and lower indices denote the regularity of the functions involved and the fact that they are compactly supported.}
\[
\ZM_{c}^{\infty}(M) = \{ \vp \in C_{c}^{\infty}(M) : \int_{M} \vp =0 \}
\]
where, obviously, the compact support condition is for free (and removed from the notation) when $M$ is compact.
\begin{proposition}[Neumann-CZ(p)]\label{prop-Neumann}
 Let $(M,g)$ be a compact, $m$-dimensional Riemannian manifold of diameter $D$ and Ricci curvature $\ric \geq - K^{2}$, $K \geq 0$. Let also $1<p<+\infty$ be fixed. Then, in the class $\ZM^{\infty}(M)$, the validity of (\ref{CZ})(p)  is equivalent to the next variant of the Calder\'on-Zygmund inequality
\begin{equation}\label{CZ'}\tag{CZ'}
 \| \Hess(\vp) \|_{L^{p}} \leq C' \{ \|\Delta \vp \|_{L^{p}} + \|\nabla \vp \|_{L^{p}} \},
\end{equation}
in the following precise sense:
\begin{itemize}
 \item [-] If (\ref{CZ})(p) holds (at least on $\ZM^{\infty}(M)\subset$) $C^{\infty}(M)$ with a constant $C>0$ then (\ref{CZ'})(p) is satisfied on $\ZM^{\infty}(M)$ with a constant $C'=C'(C,m,p,K,D)>0$.
 \item [-] If (\ref{CZ'})(p) holds on $\ZM^{\infty}(M)$ with a constant $C'>0$ then (\ref{CZ})(p) holds with a constant $C=C(C',p)>0$ in the class $C^{\infty}(M)$ (and in particular on $\ZM^{\infty}(M)$).
\end{itemize}
\end{proposition}
\end{remark}

With this in mind, the  proof of Theorem \ref{th-DZ} goes as follows.
 \begin{proof}[Proof of Theorem \ref{th-DZ} (sketch)]
According to Proposition \ref{prop-Neumann} it is enough to prove that the constant in (\ref{CZ'})(p), $p >m$, cannot depend only on $m,p$ and on a lower bound of the sectional curvature.

The starting point is the existence of a sequence of compact, $m$-dimensional Riemannian manifolds $(M_{k},g_{k})$ with $\diam M_{k} \leq D$ and $\sect_{M_{k}} \geq 0$ that converges in the measured Gromov-Hausdorff topology to a non-collapsed Alexandrov space $(X^{m},d,\H^{m})$ satisfying the following conditions
\begin{itemize}
 \item [(a)] $\mathrm{Curv}(X) \geq 0$ in the sense of Alexandrov.
 \item [(b)] $X$ has a dense set of {\it sharp singularities}, in the sense that the tangent cones at each point of this dense set are singular cones.
\end{itemize}
The limit space is the boundary of a convex set in $\rr^{m+1}$ and is constructed in \cite[Example (2)]{OS-JDG}. The existence of the approximating sequence is observed e.g. in \cite[Theorem 1]{AKP-Illinois}. 

The proof now proceeds by contradiction. Assume that \eqref{CZ'}(p), $p>m$, holds with a constant depending only on $m,p$ and on the lower bound of the sectional curvature. Then, using the Morrey-Sobolev embedding theorem we deduce the existence of a uniform constant $E=E(m,p,D)>0$ such that, for every $k$ and for every $u\in C^{\infty}(M_{k})$,
\begin{equation}\label{DZ3}
| |\nabla u |(p) - |\nabla u |(q) | \leq E (\|\Delta u\|_{L^{p}} + \| \nabla u \|_{L^{p}}) d_{k}(p,q)^{\frac{p-m}{p}}\, \text{ on } \, M_{k}.
\end{equation} 
Now, according to Appendix \ref{section-singularPoisson}, we consider a nonconstant $W^{1,2}\cap L^{p}$-solution  $u_{\infty} : X \to \rr$ of the Poisson equation $\Delta_{X} 
u_{\infty} = g_{\infty}$ with $g_{\infty} \in L^{p}(X)$, and we assume  that there exists a sequence of solutions $u_{k}:M_{k} \to \rr$ of  $\Delta_{g_{k}} u_{k} = g_{k} \in L^{p}$ on $M_{k}$ satisfying the following conditions:
\begin{itemize}
 \item $u_{k} \to u_{\infty}$ strongly in $W^{1,2}$;
 \item $\Delta u_{k} = g_{k} \to g_{\infty} = \Delta u_{\infty}$ strongly in $L^{p}$, hence, $\| \Delta u_{k} \|_{L^{p}} \leq C_{1}$;
 \item $\| u_{k} \|_{W^{1,p}} \leq C_{2}$.
\end{itemize}
Using \eqref{DZ3} along $u_{k}$ we get
\[
| |\nabla u_{k} |(p) - |\nabla u_{k} |(q) | \leq C d_{k}(p,q)^{\frac{p-m}{p}}
\]
for some constant $C>0$ independent of $k$, and then, passing to the limit, we deduce that $|\nabla u_{\infty}|$ is a continuous function; see \cite[Proposition 3.3]{Ho-Crelle}. On the other hand, as a consequence of  \cite[Theorem 1.1]{DZ-preprint}, we have that $|\nabla u_{\infty}(x)|=0$ at sharp singular points of the Alexandrov space $X$. Since this set is dense in $X$, $u_{\infty}$ must be constant, a contradiction.
\end{proof}

There is a third category of counterexamples that rely on the lack of compact supported approximation of Sobolev functions, as alluded to in Section \ref{section-density}. In fact, we have the following simple, although nontrivial, observation; \cite{Ve-Counterexample}.
\begin{proposition}\label{prop-density}
 Let $(M,g)$ be a complete, $m$-dimensional Riemannian manifold and let $1< p < +\infty$ be fixed. Then, keeping the notation of Section \ref{section-density}, the following implication holds:
 \[
(\ref{CZ})(p)\, \Longrightarrow \, W^{2,p}(M) = W^{2,p}_{0}(M) =\tW^{2,p}_{0}(M) = \tW^{2,p}(M).
 \]
 This implies that, given $u \in L^{p}(M)$ with $\Delta u \in L^{p}(M)$ it holds also that $| \nabla u| , |\Hess(u)| \in L^{p}(M)$ and, hence, there exists a sequence of cut-off functions $\vp_{k}\in C^{\infty}_{c}(M)$ such that
 \[
 \vp_{k} \overset{W^{2,p}}{\to} u.
 \]
\end{proposition}
We make a crucial use of the following surprising result by O. Milatovich, \cite[Appendix A]{GP-AMPA}, extending Strichartz's to the $L^{p}$-scale.

\begin{theorem}\label{th-Milatovich}
 Let $(M,g)$ be a complete Riemannian manifold of dimension $\dim M =m$ and let $1<p<+ \infty$ be fixed. Then,
 \[
\tW^{2,p}_{0}(M) = \tW^{2,p}(M).
\]
\end{theorem}
 
\begin{proof}[Proof of Proposition \ref{prop-density}]
According to Theorem \ref{th-Milatovich},
\[
\tW^{2,p}_{0}(M) = \tW^{2,p}(M).
\]
On the other hand, by the Cauchy-Schwarz inequality, $|\Delta u | \leq \sqrt{m} |\Hess (u)|$ and, therefore,
\[
W^{2,p}_{0}(M) \subseteq \tW^{2,p}_{0}(M).
\]
To conclude the validity of the opposite inclusion
\[
\tW^{2,p}_{0}(M) \subseteq W^{2,p}_{0}(M)
\]
observe that, by (\ref{CZ})(p), $\| \vp \|_{W^{2,p}} \leq C \| \vp \|_{\tW^{2,p}}$ for every $\vp \in C^{\infty}_{c}(M)$. This proves that
\[
\tW^{2,p}(M) = \tW^{2,p}_{0}(M)  = W^{2,p}_{0}(M) \subseteq W^{2,p}(M).
\]
Finally, using the Meyers-Serrin type result, Theorem \ref{th-MS}, joint with Cauchy-Schwarz again, yields that
\[
W^{2,p}(M) \subseteq \tW^{2,p}(M)
\]
and the chain of equalities is completed.
\end{proof}

Accordingly, we have the following new counterexample.

\begin{corollary}
 Let $(M,g)$ be a complete Riemannian manifold where, for some $1<p<+\infty$, the inclusion of $W_{0}^{2,p}(M)$ in $W^{2,p}(M)$ is strict. We know from \cite{Ve-Counterexample} that, at least for $p \geq 2$, such a manifold exits. Its volume is finite and its curvature growths more than quadratically (like $r^{4}$). Then, on this manifold, (\ref{CZ})(p) must be violated.
\end{corollary}

\begin{remark}
The previous discussion shows that finding counterexamples to second order density results is  harder than finding counterexamples to Calder\'on-Zygmund inequalities. On the other hand, from the purely Calder\'on-Zygmund viewpoint, a nice feature of Example \ref{example-L2} and Theorem \ref{th-Li} is that the violating sequence of compactly supported functions is rather explicit. This is useful in some instances, see again \cite{Ve-Counterexample}.
\end{remark}

\section{$L^{p}$-Hessian estimates in $C^{\infty}_{c}$: the local-to-global approach} \label{section-localtoglobal}

In order to extend on a manifold, with appropriate adjustments, the classical integral inequalities enjoyed by the Euclidean space (or by other model geometries) there is a classical strategy that moves from local considerations to the desired global result via a gluing procedure.\smallskip

When specified to Calder\'on-Zygmund, this strategy is articulated in the following steps.

\begin{itemize}
\item [(A)]  On ``small'' balls of a complete manifold, where the  metric is $C^{1}$  controlled, we have the validity of \eqref{CZ}(p):
 	\begin{itemize}
	\item [(A.1)] either by {\it transplanting} the Euclidean estimate (soft)
	\item [(A.2)] or by proving directly \eqref{CZ}(p) on the manifold (hard).
	\end{itemize}

 \item [(B)] Lifting Euclidean inequalities may give rise to first order (gradient) terms that we have to get rid of.
 
 \item [(C)] If the intersection multiplicity  of the balls is uniformly controlled, then, we can sum up the local inequalities and conclude the validity of the global one.
\end{itemize}

\begin{remark}
This method is robust enough to encompass Lipschitz maps $u: (M,g) \to (N,g)$ between complete manifolds with positive harmonic radii. As we shall see in Section \ref{section-maps}, this gives rise to a quantitative {\it nonlinear} \ref{CZ}(p) that, in case of compact manifolds, can be applied to isometric immersions in the same spirit of Section \ref{section-isometric-euclidean}.
\end{remark}

Let's implement the strategy step by step.

\subsection{Controlling the metric coefficients in small balls}\label{section-radii}

In a coordinate system $\phi = (x^{1},\cdots,x^{m})$ where each coordinate function is harmonic, the local expression of the Laplace-Beltrami operator is purely second order and displays a $0^{th}$ order dependence on the metric coefficients. This follows from the fact that $0 = \Delta x^{j} = g^{ik} \Gamma^{j}_{ki}$ where, we recall, $\Gamma^{i}_{ki}$ denote the Christoffel symbols with respect to $\phi$. Nowadays, it is very well understood that, generically, in these coordinates we have the maximal regularity control on the metric coefficients under the minimal amount of control of the curvature tensor. The game is then to estimate from below, in terms of the geometry, the radius of the ball where such coordinates are defined. This very classical topic goes under the name of {\it estimates of the harmonic radius}. Beside the original papers by J. Jost and H. Karcher, \cite{JK-Manuscripta}, M. Anderson, \cite{An-Invent}, and Anderson and J. Cheeger, \cite{AC-JDG}, we refer the reader to \cite{HH-Roma} for a panoramic view on the subject with applications to convergence theory.
Due to its crucial role in what we are going to do, we take some time to define everything with some degree of detail. Let us start by introducing the following

\begin{definition}[a pletora of radii]
Let $(M,g)$ be a Riemannian manifold and let $x \in M$ be a fixed reference point.\smallskip

\noindent {\rm - } The {\bf conjugate radius at $x$}  is the supremum of all $R>0$ such that the exponential map $\exp_{x}: \bb_{R}(0)\subset T_{p}M \to B_{R}(x)$ is non-singular, hence an immersion. We denote this radius by $r_{\conj}(x)$.\smallskip

\noindent {\rm - } The {\bf injectivity radius at $x$}  is the supremum of all $R>0$ such that the exponential map $\exp_{x}: \bb_{R}(0)\subset T_{p}M \to B_{R}(x)$ is a diffeomorphism.  We denote this radius by $r_{\inj}(x)$ and the corresponding coordinates are called {\bf normal coordinates at $x$}.\smallskip

\noindent {\rm - }  The {\bf Euclidean radius at $x$} is the supremum of all $R>0$ such that there exists a coordinate chart $\phi : B_{R}(x) \to \rr^{m}$ satisfying
\begin{itemize}
 \item [a)] $2^{-1} [\delta_{ij}] \leq [g_{ij}] \leq 2 [\delta_{ij}]$.
\end{itemize}
We denote this radius by $r_{\euc}(x)$. \smallskip

\noindent {\rm - }  The {\bf $C^{0,\a}$-harmonic radius at $x$}, $\a < 1/m$, is the supremum of all $R>0$ such that there exists a coordinate chart $\phi : B_{R}(x) \to \rr^{m}$ satisfying
 \begin{itemize}
 \item [a)] $2^{-1} [\delta_{ij}] \leq [g_{ij}] \leq 2 [\delta_{ij}]$;\smallskip
 \item [b)] $R^{ 1 - m \a}\| g_{ij} \|_{\a,B_{R}(x)} \leq 1$;\smallskip
 \item [c)] $\phi$ is a harmonic map, i.e., $g^{ij} \Gamma^{k}_{ij} =0$ and thus $\Delta = g^{ij}\partial_{ij}$.
\end{itemize}
We denote this radius by $r_{\harm,C^{0,\a}}(x)$.\smallskip

\noindent {\rm - }  The {\bf $C^{k,\a}$ harmonic radius at $x$}, $k \geq 1$, is the supremum of all $R>0$ such that there exists a coordinate chart $\phi : B_{R}(x) \to \rr^{m}$ satisfying
\begin{itemize}
 \item [a)] $2^{-1} [\delta_{ij}] \leq [g_{ij}] \leq 2 [\delta_{ij}]$;\smallskip
 \item [b$_{1}$)] $\sum_{1 \leq |J| \leq k}R^{|J|}\| \partial^{J} g_{ij} \|_{0,B_{R}(x)} \leq 1$;\smallskip
 \item [b$_{2}$)] $R^{k+\alpha}\| \partial^{J} g_{ij} \|_{{\alpha},B_{R}(x)} \leq 1$, for all $|J|=k$;\smallskip
 \item [c)] $\phi$ is a harmonic map.
\end{itemize}
We denote this radius by $r_{\harm,C^{k,\a}}(x)$.

\noindent {\rm -} The {\bf $W^{k,p}$ harmonic radius at $x$}, $k\geq 1$ and $1<p<+\infty$, is the supremum of all $R>0$ such that there exists a coordinate chart $\phi : B_{R}(x) \to \rr^{m}$ satisfying
\begin{itemize}
  \item [a)] $2^{-1} [\delta_{ij}] \leq [g_{ij}] \leq 2 [\delta_{ij}]$;\smallskip
  \item [b)] $\sum_{1 \leq |J| \leq k}R^{|J|-m/p}\| \partial^{J} g_{ij} \|_{L^{p}(B_{R}(x))} \leq 1$;\smallskip
  \item [c)] $\phi$ is a harmonic map.
\end{itemize}
We denote this radius by $r_{\harm,W^{k,p}}(x)$.
\end{definition}

\begin{remark}
 All these radii are positive and depend continuously on the reference point. In particular, if $M$ is compact,
 \[
 r_{\bullet}(M):=\inf_{x\in M} r_{\bullet}(x) >0.
 \]
 As a matter of fact, the Euclidean and the harmonic radii are even Lipschitz continuous. This is not generally true for the injectivity radius and the reason is the possible presence of conjugate points. We shall come back to this important fact in a moment.
\end{remark}
\begin{remark}
There is a hierarchy between  the harmonic radii collected in the above definition. Accordingly, in terms of the geometry, a lower estimate of $r_{\harm, C^{k,\a}}$ is more demanding  than an estimate of $r_{\harm, C^{0,\a}}$ which in turn, by Sobolev embeddings, is implied by a lower estimate of $r_{\harm,W^{k,p}}$. The Euclidean radius,  introduced by B. G\"uneysu in \cite{Gu-book, BG-preprint}, is the roughest of this family of radii as it  captures only the local bi-Lip content of the metric. Nevertheless, working within the Euclidean radius is enough to get interesting heat kernel estimates and therefore it would be important to understand how much geometry must be controlled to get a lower estimate. Needless to say, every time we are able to control the $C^{0,\a}$-harmonic radius (a lower Ricci bound joint with a lower injectivity radius bound is enough) we can control the Euclidean radius but this looks too much demanding. The example presented in Appendix \ref{appendix-euclidean} shows that controlling the injectivity radius is not enough.
\end{remark}
By combining  curvature bounds with injectivity radius estimates gives a uniform lower control on the harmonic (and Euclidean) radius, where the metric coefficients are close, in the above specified sense, to the Euclidean's. We shall need the following important result concerning the $C^{1,\a}$-harmonic radius.
\begin{theorem}[Jost-Karcher, Anderson]\label{th-JK-A}
 Let $B_{2r}(x)$ be a compact ball in the Riemannian manifold $(M,g)$. Assume that either one of the following conditions is satisfied:
\begin{itemize}
 \item [(a)]  Let $\frak{i}(x):= \min( r , r_{\inj}(x))>0$ and assume that $r_{\conj}(y) \geq \frak{i}(x)$ for every $y \in B_{\frac{\frak{i}(x)}{2}}(x)$.\smallskip
 \item [(b)] Let $\displaystyle \frak{i}(x) := \min (r , \inf_{y \in B_{r}(x)} r_{\inj}(y))$ and assume that $\frak{i}(x)>0$.
\end{itemize}
Finally, let $R (x)= \| \ric \|_{L^{\infty}(B_{2r}(x))}$. Then, there exists $H=H(r,R(x),\frak{i}(x),m,\a)>0$ such that
\[
r_{\harm,C^{1,\a}}(x ) \geq H.
\]
\end{theorem}

\begin{remark}\label{rem-harmonicradius1}
As it is stated, case (a) of Theorem \ref{th-JK-A} may look slightly unconventional. It represents an abstract version of Jost-Karcher estimate in \cite{JK-Manuscripta}. In fact, note that:
\begin{itemize}
\item The peculiarity of case (a) is that, from the viewpoint of the injectivity radius, its value at a single point is enough to have an estimate of the harmonic radius at that point. This is especially relevant when the conjugate locus is empty.
 \item If $\sect \leq S^{2}$ then, by Rauch comparison, $r_{\conj}(y) \geq \pi/S$. It follows that Theorem \ref{th-JK-A} (a) applies when $| \sect | \leq S^{2}$ by choosing  $r \leq \pi/S$. This is the original formulation by Jost-Karcher.
 \item If $\sect \leq S^{2}$,  $r \leq \pi/S$, and $\bb_{r}(0) \subseteq T_{x}M$ is endowed with the pull-back metric $\hat g = \exp_{x}^{\ast}g$ then $\hat r_{\inj}(0) \geq r$. It follows that, when $| \sect | \leq S^{2}$, Theorem \ref{th-JK-A} (a) applies to $(\bb_{r}(0),\hat g)$ and gives a harmonic radius bound $H>0$ depending only on $S,m$ and $\a$. This has many applications in obtaining local uniform estimates of solutions of PDEs under sectional curvature assumptions.  In this respect, we suggest the reader to take a look at the illuminating paper \cite{Ca-Pisa} by G. Carron.
\end{itemize}
 \end{remark}

\begin{remark}
Case (a) of Theorem \ref{th-JK-A} reduces to (b) thanks to the following  result by S. Xu, \cite{Xu-CCM}, showing that, in a very precise sense, the obstruction for $r_{\inj}$ to be $1$-Lipschitz continuous is inside the conjugate locus.
\begin{theorem}
Let $(M,g)$ be a Riemannian manifold. Given $x \in M$,
\[
r_{\inj}(y) \geq \min (r_{\inj}(x),r_{\conj}(y)) - d(x,y),\quad \forall y \in B_{\frac{r_{\inj}(x)}{2}}(x).
\]
In particular, if $M$ is complete and does not contain conjugate points, then the function $x \mapsto r_{\inj}(x)$ is $1$-Lipschitz.
\end{theorem}
\end{remark}
 
\begin{remark}
 Case (b) of Theorem \ref{th-JK-A} is the contribution of Anderson, \cite{An-Invent}, who first realized that a sectional curvature absolute bound can be replaced by a Ricci curvature absolute bound\footnote{in the language of A. Naber we are switching from ``Hessian'' to ``Laplacian'' bounds.} up to paying the prize of imposing an injectivity radius constraint not only at the needed point but in a neighborhood of this incriminated point. Equivalently, switching from sectional to Ricci curvature implies that the pointwise harmonic radius estimate depends on the behaviour of injectivity radius in a neighborhood of the point.
\end{remark}

\subsection{Transplanting the Euclidean inequality}\label{section-localinequality1}

The origin of everything contained in these notes is the celebrated Calder\'on-Zygmund inequality in the Euclidean space $\rr^{n}$. A standard reference, for those oriented in PDEs, is D. Gilbarg and N. Trudinger bible; see \cite[Chapter 9]{GT}. Let $1<p<+\infty$. Then, there exists a constant $C =C(n,p)>0$ such that
\[
 \| \Hess(u) \|_{L^{p}} \leq C \cdot \| \Delta u \|_{L^{p}}
\]
holds for every $u \in C^{\infty}_{c}(\rr^{n})$. The classical proof starts from the representation formula
\[
\frac{\partial^{|J|}}{\partial x^{J}}u(x) = - \int_{\rr^{n}} \frac{\partial^{|J|}}{\partial x^{J}}G(x-y) \Delta u(y) dy,
\]
where $G$ is the (possibly signed) Green kernel of $\rr^{n}$, and consists in obtaining estimates of the corresponding singular integrals. A completely new approach, based on iterative estimates of the distribution function of the Hardy-Littlewood maximal function of $|\Hess(u)|^{2}$ is due to L. Caffarelli, \cite{Ca-Annals, Ca-Milan}. This latter proof, that works for fully nonlinear operators, later inspired  L. Wang, \cite{Wa-ActaSinica}, who provided an argument that looks very suitable to be adapted on a Riemannian manifold. See Section \ref{section-localinequality2}.\smallskip

The Calder\'on-Zygmund  inequality for the Laplacian actually extends to a uniformly elliptic operator of second order  $L$ and localizes on compact domains without any boundary conditions. More precisely, given $\O \Subset \rr^{n}$, let $L =  a^{ij}\partial^{2}_{ij}$ with $\|a^{ij}\|_{C^{1} }(\O) \leq A$, $a^{ij}v^{i}v^{j} \geq \l |v|^{2}$, for some constants $A,\l>0$. Accordingly, for any domain $\O_{1} \Subset \O_{2} \Subset \O$, there exists a constant $C>0$ depending on $\Omega_{1},\Omega_{2}$, $p$, $m$, the ellipticity constant $\l>0$ and the $C^{1}$-bound $A$, such that
\[
\| \Hess(u) \|_{L^{p}(\O_{1})} \leq C \{\| u \|_{L^{p}(\O_{2})} + \| Lu\|_{L^{p}(\O_{2})} \}
\]
holds for every $u \in C^{\infty}(\O_{2})$. Since, in harmonic coordinates,
\begin{itemize}
 \item $\Hess(u)_{ij} = \partial^{2}_{ij}u - \Gamma_{ij}^{k}\partial_{k} u$ with $\Gamma_{ij}^{k}=\Gamma(g,\partial g)$\smallskip
 \item $\Delta u = g^{ij}\partial^{2}_{ij}u$\smallskip
 \item $|\nabla u|^{2} = g^{ij}\partial_{i} u \, \partial_{j} u$,\smallskip
 \item $\dv=\sqrt{\det g}\, \dx$
\end{itemize}
from these considerations we obtain that, within the $C^{1,\a}$ harmonic radius of $x \in M$, the Euclidean inequality transplants to $M$.

\begin{lemma}[Euclidean CZ transplanted]\label{lemma-transplanted}
 Let $x \in M$ and $0< 2r <r_{H}(x):=r_{\harm,C^{1,\a}}(x)$. Then, for any $u$ on $B_{r_{H}}(x)$:
 \[
C^{-1} \| 1_{B_{r}(x)} \Hess(u) \|_{L^{p}} \leq \| 1_{B_{2r}(x)} u \|_{L^{p}} + \| 1_{B_{2r}(x)} \Delta u \|_{L^{p}} + \|1_{B_{2r}(x)} \nabla u \|_{L^{p}}
 \]
 for some constant $C=C(m,r_{H}(x),p)>0$. 
\end{lemma}
In the sequel, in order to glue local inequalities together, we will ask $r_{H}$ to be uniform in $x$. This will be done by requiring suitable geometric conditions as explained in Theorem \ref{th-JK-A}.

\subsection{Local inequality without transplantation}\label{section-localinequality2}

Transplanting is the easiest but not the unique way to produce local \eqref{CZ}(p) on manifolds. If we have a control on the Sectional curvature, using L. Wang ideas we can prove the following partial result. See \cite{GMP-progress}.
 
 \begin{theorem}\label{th-nontrasplanted}
 Let $(M,g)$ be a complete $m$-dimensional Riemannian manifold.  Fix $\bx \in M$ and let $S=S(\bx)>0$  and $D = D(\bx)>0$ be s.t.
\[
|\sect| \leq S, \quad \text{on }B_{4}(\bx)
 \]
 and
\[
\vol B_{2R}(x)\, \leq D \cdot  \vol B_{R}(x) ,\,\, \forall B_{2R}(x) \Subset B_{4}(\bx).
\]
Let $p \in (m/2,+\infty) \cap [2,+\infty)$. Then, there exists a constant $C = C(D,S,m,p)>0$  such that the following inequality
\begin{align*}
 \int_{B_{1}(\bx)} | \Hess(u) |^{p}  \leq  C \left\{ \int_{B_{2}(\bx)} |u|^{p}  +  \int_{B_{2}(\bx)} |\Delta u|^{p} \right\}.
\end{align*}
holds for every $u \in C^{2}(B_{4}(\bx))$.
\end{theorem}

 \begin{remark}
 Some observations on the statement are in order.
\begin{itemize} 
 \item [(a)] The large exponent is unnatural and depends on mean value inequalities. We expect to be able to lower it.
 \item [(b)] The sectional curvature bound is used to get locally uniform Schauder estimates. These are obtained from the Euclidean's in the lifted ball $(\bb_{\pi/\sqrt{S}}(0) , \hat g)$ of $T_{\bx}M$ endowed with the Riemannian metric $\hat g = \exp_{x}^{\ast}g$. This is possible because $L^{\infty}$ bounds of solutions of PDEs are stable by local isometries. Moreover, as we have already observed, Jost-Karcher estimate of the harmonic radius applies in this situation; see Theorem \ref{th-JK-A} and Remark \ref{rem-harmonicradius1}. Finally, in harmonic coodinates, Schauder estimates depend only on the $C^{1}$-norm of $g^{ij}$.
 \item [(c)] Due to the previous observation one may be tempted to get directly the $L^{p}$ estimate using the lifting metric. This is forbidden because of no control on the number of sheets of the local covering $\exp_{x} : \bb_{\pi/\sqrt{S}}(0) \to B_{\pi/\sqrt{S}}(x)$ as $x$ moves in $M$. For instance this problem happens when $M = \hh^{m}/_{\Gamma}$ with $\vol(M)<+\infty$, \cite{GMP-progress}.
\end{itemize}
\end{remark}

\begin{proof}[Outline of the Proof]
We follow closely \cite{Wa-ActaSinica}. Fix $p \geq 2$ (large enough) and let $u \in C^{2}( B_{12}(\bx))$ be a solution of $\Delta u = f$. We have to estimate the $L^{p}$ norm of $|\Hess(u)|$ on $B_{1}(\bx)$. The idea is to reduce the problem to estimating (on a larger ball) the distribution function $\w_{\FM(|\Hess(u)|^{2})}(x)$ of the maximal function $\FM(|\Hess(u)|^{2})(x)$ of the pointwise square norm  $|\Hess(u)|^{2}(x)$ of the Hessian of $u$.

Indeed, since $p \geq 2$, for any locally integrable function $g$ one has
\begin{equation}
\int_{B_{1}(\bx)} |g|^{p}  = \int_{B_{1}(\bx)} (g^{2})^{\frac{p}{2}}  \leq \int_{B_{1}(\bx)} \FM(|g|^{2})^{\frac{p}{2}} 
\end{equation}
and, for any fixed $\CN>1$, using the Lebesgue-Stieltjies integral, it holds
\begin{align}\label{outline-2}
  \int_{B_{1}(\bx)} \FM(g^{2})^{p/2}  &= p \int_{0}^{+\infty} t^{p-1} \w_{\FM(g^{2})}(t^{2}) \,\mathrm{d}t  \\ \nonumber
  &= p \sum_{k=0}^{+\infty} \int_{\CN^{k}}^{\CN^{k+1}} t^{p-1} \w_{\FM(g^{2})}(t^{2}) \mathrm{d}t + p\int_{0}^{1}  t^{p-1} \w_{\FM(g^{2})}(t^{2}) \mathrm{d}t \\ \nonumber
 &\leq p  \CN^{p-2}(\CN-1)\sum_{k=1}^{+\infty} \CN^{pk}  \w_{\FM(g^{2})}(\CN^{2k}) + p \,  \vol B_{1}(\bx) .\\ \nonumber
\end{align}
Now, suppose we have rescaled $u$ to $u / \l$ in such a way that
$ \vol(\{ \FM(|\Hess(u)|^{2}) > \CN^{2} \}) < \e  \vol B_{1}(\bx)$ where the stretching factor $\l>0$ is suitably chosen (depending on the local geometry and on $\|u\|_{L^{p}}+\| f \|_{L^{p}})$. Then, using a mixture of PDEs theory and  a version of the the Vitali covering Lemma, it is shown that, for a suitable choice of $\d>0$ depending on the local geometry and on $\e$, one has the following estimate:
\begin{align}
 \w_{\FM(|\Hess(u)|^{2})}(\CN^{2k}) &\leq \sum_{i=1}^{k} \e^{i} \w_{\FM(f^{2})}(\d^{2} (\CN^{2})^{k-i} ) + \e^{k} \w_{\FM(|\Hess(u)|^{2})} (1) \\ \nonumber
 &\leq  \sum_{i=1}^{k} \e^{i} \w_{\FM(f^{2})}(\d^{2} (\CN^{2})^{k-i} ) + \e^{k}  \vol B_{1}(\bx).
\end{align}
This is the hard analytic part of the proof. Summarizing we have obtained:
\begin{align}\label{outline-3}
 \int_{B_{1}(\bx)} |\Hess(u)|^{p} 
 &\leq \CC_{1} \sum_{k=0}^{+\infty} \CN^{kp} \sum_{i=1}^{k} \e^{i} \w_{\FM(f^{2})}(\d^{2} (\CN^{2})^{k-i} ) + \CC_{2}  \vol B_{1}(\bx) \\ \nonumber
 &= \CC_{1} \sum_{i=1}^{+\infty} \e^{i} \CN^{ip} \sum_{k=i}^{+\infty} \w_{\FM(f^{2})}(\d^{2} (\CN^{2})^{k-i} ) + \CC_{2}  \vol B_{1}(\bx)  \\ \nonumber
 &= \CC_{1} \sum_{i=1}^{+\infty} \e^{i} \CN^{ip} \sum_{k=0}^{+\infty} \w_{\FM(f^{2})}(\d^{2} \CN^{2k} ) + \CC_{2}  \vol B_{1}(\bx)\\ \nonumber
\end{align}
provided $\e>0$ is small enough that $\e\CN^{p}<1$  (so  to have the convergence of the geometric series). Here, we have set
\begin{equation}
 \CC_{1} \simeq \CN^{p-2}(\CN-1),\quad \CC_{2} \simeq \CN^{p-2}(\CN-1).
\end{equation}
On the other hand, using again the Stieltjies integral as in \eqref{outline-2}, we easily obtain the lower estimate
\begin{align}
 \int_{B_{1}(\bx)} \FM(f^{2})^{\frac{p}{2}}  &\geq \frac{p \d^{p}(\CN-1)}{\CN^{p}} \sum_{k=1}^{+\infty} \CN^{kp} \w_{\FM(f^{2})} ( \d^{2}\CN^{2k}).\nonumber
\end{align}
Inserting into \eqref{outline-3} this latter  gives
\begin{align}
 \int_{B_{1}(\bx)} |\Hess(u)|^{p}  &\lessapprox \frac{\CN^{2(p-1)}}{\d^{p}}  \int_{B_{1}(\bx)} \FM(f^{2})^{\frac{p}{2}} + \CN^{p-2}(\CN-1)  \vol B_{1}(\bx).
\end{align}
Whence, recalling the strong $q-q$ estimate of the maximal function with $q= \frac{p}{2}>1$, we conclude
\begin{equation}
 \int_{B_{1}(\bx)} |\Hess(u)|^{p}  \lessapprox \frac{\CN^{2(p-1)}}{\d^{p}} \int_{B_{6}(\bx)} |f|^{p} +\CN^{p-2}(\CN-1)  \vol B_{1}(\bx) . 
\end{equation}
After scaling back the function $u$  this latter writes  in the form of an $L^{p}$ Calder\'on-Zygmund inequality.
\end{proof}

\subsection{Getting rid of the gradient term}\label{section-interpolation}

Recall that, as a drawback of the  transplantation procedure, an $L^{p}$-norm of the gradient of the function appeared. The genuine $L^{p}$ control of the gradient in terms of the function and its Laplacian will be discussed in the last section of the paper; see Section \ref{section-gradient}. What we need here is just  the {\it the interpolation inequalities} collected in the next Lemma. As we shall see in a moment, it also has interesting ``abstract'' consequences.

\begin{lemma}\label{lemma-interpolation}
Let $(M,g)$ be an $m$-dimensional Riemannian manifold and let $1<p<+\infty$ be a fixed exponent.\smallskip

\noindent (A) Assume $2\leq p<\infty$. Then, there is a constant $C=C(p)>0$ such that for any $\varepsilon>0$, one has
\begin{align} \label{interpolation1} \tag{$\I _{p\geq 2}$}
\|  \nabla \vp \| _{L^{p}} \leq    \frac{C}{\varepsilon}\|  \vp \| _{L^{p}} +C\varepsilon\| 
\mathrm{Hess}\left(  \vp \right)  \| _{L^{p}},
\end{align}
for every $\vp \in C^{\infty}_{\mathrm{c}}(M)$.\smallskip

\noindent (B) Assume $1<p \leq 2$ and that $M$ is complete with possibly nonempty boundary $\partial M \not= 
\emptyset$. Then, there is a constant $C=C(p)>0$ such that for any $\varepsilon>0$, one has
\begin{align}\label{interpolation2} \tag{$\I _{p \leq 2}$}
\|  \nabla \vp \| _{L^{p}} \leq   2C \| \vp \|_{L^{p}}^{\frac 1 2} \| \Delta \vp \|_{L^{p}}^{\frac 1 2} \leq \frac{C}{\varepsilon}\|  \vp \| _{L^{p}} +C\varepsilon\| 
\Delta  \vp  \| _{L^{p}},
\end{align}
for every  $\vp \in C^{\infty}_{\mathrm{c}}(M)$.
\end{lemma}

The proof of \eqref{interpolation1} is elementary and relies on integration by parts and on Young inequalities, \cite{GP-AdvMath}. The multiplicative inequality \eqref{interpolation2}, instead, is a contribution of T. Coulhon and X.-T. Duong, \cite{CD-CPAM}, and it is obtained via heat-kernel estimates. Actually, in \cite{CD-CPAM} only the boundaryless case is considered. The presence of the boundary can be overcome using the complete Riemannian extension construced in \cite[Theorem A]{PV-AnnaliPisa}. In \cite{CD-CPAM} it is also observed that  \eqref{interpolation2} can be extended to the whole $L^{p}$-scale on complete manifolds that, essentially, have nonnegative Ricci curvature. We shall see in Section \ref{section-gradient} that a lower Ricci curvature bound is enough to extend \eqref{interpolation2} (but not in the form of a multiplicative inequality)  to the space $\tW^{2,p}(M)$ for every $1<p<+\infty$.\smallskip

By combining Lemma \ref{lemma-interpolation} with Lemma \ref{lemma-transplanted}, and using also the harmonic radius estimates of Theorem \ref{th-JK-A} we get the following conclusion.

\begin{corollary}\label{cor-transplanted-nogradient}
Let $(M,g)$ be an $m$-dimensional complete Riemannian manifold satisfying $\| \ric \|_{L^{\infty}}< R$ and $r_{\inj}(M) = \frak{i}>0$. Then,  there exists a radius $r_{H} = r_{H}(m,R,\frak{i})>0$ and a constant $C=C(m,r_{H},p)>0$ such that the following Calder\'on-Zygmund inequality
 \[
 \| 1_{B_{r}(x)} \Hess(u) \|_{L^{p}} \leq C \left\{ \| 1_{B_{2r}(x)} u \|_{L^{p}} + \| 1_{B_{2r}(x)} \Delta u \|_{L^{p}} \right\}
 \]
holds  on any ball $B_{r}(x)$ with $0< 2r <r_{H}$ and for every $u \in C^{\infty}(B_{r_{H}}(x))$.
\end{corollary}

Let us now point out some abstract consequences of the  interpolation inequalities. The first application gives an indication on how much $L^{p}$-gradient estimates and Calder\'on-Zygmund inequalities are closely related. For a more extensive discussion see Section \ref{section-Lpgradient}.

\begin{corollary}[$L^{p}$-estimates of the gradient in $C^{\infty}_{c}$]\label{cor-grad-compactlysupp}
If, for some $1<p<+\infty$, \ref{CZ}(p) holds on a complete Riemannian manifold $(M,g)$ then
 \[
\|  \nabla \vp \|_{L^{p}} \leq  C \cdot \left\{  \| \vp \|_{L^{p}} + \|
\Delta  \vp  \|_{L^{p}}  \right\}, \, \forall \vp \in C^{\infty}_{c}(M).
 \]
 Moreover, if we restrict the range to $1<p \leq 2$, then the above estimate holds without assuming the validity of \ref{CZ}(p).
\end{corollary}

As a second consequence we have that the validity of \eqref{CZ}(p) depends only on the geometry at infinity of the underlying manifold. A similar phenomenon appears also in the setting of Sobolev inequalities, \cite{Ca-Duke}.

\begin{corollary}[ends and connect sum]
Let $(M,g)$ be a complete, non-compact Riemannian manifold and let $1<p<+\infty$. Then
  \[
 M \text{ satisfies \ref{CZ}(p)} \Leftrightarrow \text{ the ends of }M \text{ satisfy \ref{CZ}(p)}
 \]
 In particular, given complete $m$-dimensional Riemannian manifolds $(M_{1},g_{1})$ and $(M_{2},g_{2})$
 \[
 M_{1},M_{2} \text{ satisfy \ref{CZ}(p) }\Rightarrow M_{1} \# M_{2}  \text{ satisfies \ref{CZ}(p)}.
 \]
\end{corollary}

\subsection{Conclusion: summing up the local inequalities}
One of the classical and powerful consequences of the relative volume comparison is that complete manifolds with uniform lower Ricci bounds are locally doubling spaces. This fact is crucial in many contexts starting from the celebrated Gromov pre-compactness theorem. Here, we need to record the following {\it finite multiplicity intersection property} of such spaces, see e.g. \cite{He-book}.
 \begin{lemma}\label{lemma-finitemultcovering}
 Let $(M,g)$ be a complete, $m$-dimensional Riemannian manifold. Assume that, for a radius $r_{0}>0$ there exists a (doubling) constant $D = D(r_{0})>0$ such that
 \[
 \vol B_{2r}(x) \leq D \, \vol B_{r}(x)
 \]
for every $x\in M$ and every $0<2r < r_{0}$. Then, there exists a number  $N = N(m,D,r_{0}) \in \nn$ with the following property:  having fixed $0< 2r <r_{0}$, there exists a sequence of balls $\{ B_{r}(x_{j}) \}$ that cover $M$ and such that every point of $M$ is contained in no more than $N$ balls of the family $\{ B_{2r}(x_{j})\}$.
\end{lemma}

Thanks to this property we can sum up the local inequalities obtained in Theorem \ref{th-nontrasplanted} and Corollary \ref{cor-transplanted-nogradient} and get the following global result.
\begin{theorem}\label{th-localtoglobal}
 Let $(M,g)$ be complete Riemannian manifold of dimension $\dim M =m$ and let $p \in \rr$. Assume that either one of the following set of assumptions is satisfied:
 \begin{equation}\label{ricci+inj}
 \| \ric \|_{L^{\infty}} = R <+\infty,\quad r_{\inj}(M) = \frak{i} >0, \quad p \in (1,+\infty)
 \end{equation}
 or
 \begin{equation}\label{sectional}
\|  \sect \|_{L^{\infty}} = S <+\infty, \quad  p \in [2,+\infty) \cap (m/2,+\infty).
 \end{equation}
 Then, there exists a constant $C>0$ depending on $m,p$ and either on $R,\frak{i}$ or on $S$, such that \ref{CZ}(p) holds true, namely,
  \begin{equation}
 \| \Hess(\vp) \|_{L^{p}} \leq C \left\{ \| \vp \|_{L^{p}} + \| \Delta \vp \|_{L^{p}} \right\},\quad \forall \vp \in C^{\infty}_{c}(M).
 \end{equation}
\end{theorem}
This  completes the discussion on the current status of the art on the {\it local-to-global} approach to prove global Calder\'on-Zygmund inequalities for real valued functions with compact support. As a matter of fact, as one may have already noted, the compactness of the support is used only in the interpolation inequalities whose validity, at the $C^{\infty}_{c}$-level, requires only the completeness of the manifold. On the other hand, it is true that in the bounded Ricci setting where the interpolation inequalities are applied, they hold for $L^{p}$ functions, as we shall see in Section \ref{section-gradient}.

\subsection{A word on manifold-valued maps}\label{section-maps}

According to \cite{GP-AGAG}, the {\it local-to-global} approach is robust enough to be adapted to manifold valued Lipschitz maps between complete manifolds with positive harmonic radii. What we get is a Calder\'on-Zygmund inequality for certain systems of PDEs. These inequalities contain a natural nonlinear term that, in some sense, encodes the non-flatness (= finiteness of the harmonic radius) of the target space. When the source manifold is compact and the map is isometric these inequalities extend, in a quantitative way, to general ambient manifolds and to the whole $L^{p}$-scale, what we saw in Section \ref{section-isometric-euclidean}.\smallskip

In order to state the main result of the section, we need to introduce some notation. Let $u : M \to N$ be a smooth map between Riemannian manifolds $(M,g)$ and $(N,h)$ of dimensions, respectively, $\dim M =m$, $\dim N =n$. For any given  point $\bx \in M$ we take coordinate charts $(U,\vp = (x^{1},\cdots,x^{m}))$ of $M$ centered at $\bx$ and $(V,\psi=(y^{1},\cdots,y^{n}))$ centered at $u(\bx)$ such that $u(U) \subseteq V$. We also fix the index convention $i,j,k,\cdots \in \{ 1,\cdots,m\}$ and $\a,\b,\g,\cdots \in \{ 1,\cdots,n\}$. Then, the {\it generalized Hessian} of $u$ is the tensor field along $u$ that, in local coordinates, writes as
\[
\Hess(u) = \Hess(u)^{\a}_{ij} \, dx^{i}\otimes dx^{j} \otimes  \left.\frac{\partial}{\partial y^{\a}}\right\vert_{u}\
\]
where
\begin{align*}
 \Hess(u)^{\a}_{ij} &= \partial^{2}_{ij} u^{\a} -\,^{M}\!\Gamma^{k}_{ij}\partial_{k} u^{\a} + \,^{N}\!\Gamma^{\a}_{\b \g} ( u)\, \partial_{i}  u^{\b}  \partial_{j}  u^{\g}\\
 &=  \,^{M}\!\Hess( u^{\a})_{ij}  + \,^{N}\!\Gamma^{\a}_{\b \g} ( u ) \, \partial_{i}  u^{\b}  \partial_{j}  u^{\g}.
\end{align*}
Tracing the generalizes Hessian gives rise to the {\it generalized Laplacian}, also called the {\it tension field}, of the map $u$. Again, in local coordinates,
\[
\Delta u = (\Delta u)^{\a} \left.\frac{\partial}{\partial y^{\a}}\right\vert_{u}
\]
where
\[
 (\Delta u)^{\a} = \Delta u^{\a} + g^{ij} \,^{N}\!\Gamma^{\a}_{\b \g} ( u ) \, \partial_{i}  u^{\b}  \partial_{j}  u^{\g}.
\]
Note that, in case $u: M \to N$ is an isometric immersion, then $\Hess(u)$ and $\Delta u$ are, respectively, the second fundamental form and the mean curvature vector field of the submanifold.\smallskip

Using tension fields instead of the ordinary Laplacian one is naturally led to study the corresponding Poisson equation
\[
\Delta u = G
\]
where $G = G^{\a} \left.\frac{\partial}{\partial y^{\a}}\right\vert_{u}$ is a smooth {\it vector field along the map $u$}.  If $G \equiv 0$ then $u$ is called a {\it harmonic map}.\smallskip

There is an endless literature on the existence and the regularity theory for harmonic maps. However, very little is known in the case of a nonzero tension field $G$. In fact, the only paper that we are aware of this  subject  is \cite{CJ-CAG} by W. Cheng and J. Jost. It investigates the existence of a solution of the Dirichlet problem in the homotopy class of a given boundary datum, and in the special setting of non-positively curved target spaces. Moreover, the first part of the same paper contains an $L^{2}$-Hessian inequality that looks pretty much like a (nonlinear) Calder\'on-Zygmund inequality for maps (in homotopy classes). Similarly to what happens in the real valued case, the study in the $L^{2}$ setting is helped very much by the use of the Bochner formula. \smallskip

The inequality we are going to state is different, it has a purely $C^{1,\alpha}$ dependence, it allows the source and target manifolds to be non-compact, and it works on the whole $L^{p}$-scale.
\begin{theorem}\label{th-CZmaps}
 Let $(M,g)$, $(N,h)$ be connected Riemannian manifolds and set $m:=\dim (M)$, $n:=\dim(N)$. Assume that $M$ is geodesically complete and that its $C^{1,\a}$-harmonic radius satisfies $r_{\harm, C^{1,\a}}(M) >0$.
  Assume also that $N$ is geodesically complete with $C^{1,\a}$-harmonic radius $r_{\harm, C^{1,\a}}(N) >0$. Then for every $1< p< +\infty$, there exists a constant $C=C(p,m,n,\a,r_{\euc}(M))>0$, which only depends on the indicated parameters, such that for all $0< \L \leq +\infty$, all $\L$-Lipschitz continuous maps $u \in C^{2}(M,N)$,  and any $o \in N$, one has 
\begin{align*} 
  C^{-1}\| \Hess ( u)\|_{L^{p}}  \leq    \| \Delta u\|_{L^{p}} + r^{-1}\|  d u\|_{L^{p}}+  r_{\harm, C^{1,\a}}(N)^{-1} \| d u \|_{L^{2p}}^2 +
r^{-2}\|    \dist_{N}(u , o) \|_{L^{p}},
 \end{align*} 
where we have set
\begin{equation}
r  =  \min \left(r_{\harm, C^{1,\a}}(M), \frac{r_{\harm, C^{1,\a}}(N)}{\max(\L,1)} , 1\right).
\end{equation}
\end{theorem}

The argument roughly goes as follows: Hessian and Laplacian have a 1st order dependence on the source and target metrics. If we can map balls within the harmonic radius into balls with the same property then we have a local estimate, up to using the center of the target ball as a reference origin to compute distances. This is where it is used that the map is Lipschitz. The actual quantitative dependence from the harmonic radii is a metter of scaling. Finally, we have to switch from the local estimate to the global one. This makes use of two ingredients: (a) a further observation on the fact that the reference origin can be fixed by keeping the same structure of the estimate; (b) the gluing procedure already described in the setting of real valued functions. Originally, we imposed a lower Ricci condition on the source manifold so to have a local doubling inequality that gives the existence of the  covering by balls with finite intersection multiplicity. However, this is not needed: the local doubling follows once we have a positive Euclidean radius and this is for free because it is even assumed that the $C^{1,\alpha}$-harmonic radius of $M$ is positive. I am grateful to Gilles Carron for this remark.  \smallskip

Note that, in case $\L=+\infty$ (i.e. no Lipschitz assumption is assumed on $u$) the statement of Theorem \ref{th-CZmaps} becomes nontrivial only if $r_{\harm,C^{1,\a}}(N)=+\infty$. Under this latter assumption, the presence of the ``nonlinear''  first order term  $\| d u \|_{L^{2p}}^2 $ disappears, and the Calder\'on-Zygmund inequality takes the classical form. However, having an infinite harmonic radius means precisely that $(N,h)$ is the standard Euclidean space. This is shown in the next result, whose proof follows from the very definition of $C^{1,\a}$-harmonic radius via an Ascoli-Arzelà argument; see e.g. \cite{GP-AGAG}.

\begin{proposition} \label{prop-rharm-euspace}
Let $(M,g)$ be a complete, non-compact, connected $m$-dimensional Riemannian manifold and assume that there exists some $o \in M$ and some $\alpha\in (0,1)$ such that $r_{\harm,C^{1,\a}}(o)= +\infty$. Then, $(M,g)$ is isometric to the Euclidean $\rr^m$.
\end{proposition}

It then follows that Theorem \ref{th-CZmaps} actually recovers, with a  quantitative dependence on the harmonic radius of the source, what we obtained in Theorem  \ref{th-localtoglobal} under assumption \eqref{ricci+inj}.\smallskip

To conclude the section let us specify Theorem \ref{th-CZmaps} to isometric immersions. The estimate we get, in the form it is stated, cannot be reduced to the Euclidean estimate via Nash embedding because, in general, the extrinsic geometry in Nash theorem cannot be controlled.\smallskip

\begin{corollary}
Let $(M,g)$ be compact and assume that the ambient manifold $(N,h)$ is complete, with $r_{\harm,C^{1,\a}}(N)>0$. Then, for every $1<p<+\infty$ and for every $R \geq 0$ such that $r_{\euc}(M) \geq R$, there exists a constant
\[
C=C\big(p,\dim(M),\dim(N), R\big)>0,
\]
which only depends on the indicated parameters, such that for every isometric immersion $\Psi:M\to N$ one has
\begin{equation}
C^{-1}\| \bA\|_{L^{p}}  \leq    \| \bH \|_{L^{p}} + \vol(M)^{1/p}\Big(r^{-1}+  r_{\harm,C^{1,\a}}(N)^{-1}  +
r^{-2}  \diam_{N}(\Psi (M))\Big)  ,
\end{equation}
where
\[
r  =  \min \left(r_{\harm,C^{1,\a}}(M), r_{\harm,C^{1,\a}}(N),1\right).
\]
\end{corollary}

\section{$L^{p}$-Hessian estimates in $C^{\infty}_{c}$: the functional analytic approach}\label{section-functionalanalytic}

In Euclidean spaces there is (at least?) a second classical approach to prove the Calder\'on-Zygmund inequality  \eqref{intro2} for the Laplace operator. This viewpoint has a deep functional analytic flavour and relies on the $L^{p}$-boundedness of the so called {\it Riesz transform}; see \cite[Chapter III, Section 1.3]{Ste-Singular}.\smallskip

\subsection{From Euclidean space to manifolds}

Let $(M,g)$ be a complete Riemannian manifold of dimension $\dim M =m \geq 2$. Take the self-adjoint realization of the square root $(-\Delta)^{\frac{1}{2}} : C^{\infty}_{c}(M) \subset L^{2}(M) \to L^{2}(M)$ of the positive definite Laplace-Beltrami operator $-\Delta$; see \cite{Str-JFA}.
Using the very definition and integration by parts, we get
\[
\| d \vp \|_{L^{2}}^{2} = \int_{M} \vp  (-\Delta \vp) = \| (-\Delta)^{\frac{1}{2}} \vp \|_{L^{2}}^{2}.
\]
On noting that this chain of equalities can be re-written in the form
\[
\| d (-\Delta)^{-\frac{1}{2}} \vp \|_{L^{2}} = \| \vp \|_{L^{2}}
\]
one is led to consider the operator
\[
\R = d (-\Delta)^{-\frac{1}{2}} : C^{\infty}_{c}(M) \subset L^{2}(M)  \to L^{2}\Lambda^{1}(M)
\]
where $\Lambda^{1}(M)$ is the bundle of  $1$-forms endowed with the compatible metric  and connection inherited from $g$. The previous equality tells us that $\R$ extends to a bounded operator on $L^{2}(M)$.\smallskip

Suppose now that $(M,g)$ is the standard Euclidean space $(\rr^{m},g_{E})$. Then, it is customary to introduce $\R$ in the context of {\it singular integrals} via the kernel
\[
K(x) = -c_{m} |x|^{-m+1},
\]
where $c_{m}>0$ is a dimensional constant. Thus, up to identifying $1$-forms and vector fields,
\[
\R( \vp )(x) = \lim_{\e \to 0} \int_{|y|>\e} \nabla K(y) \vp(x-y) \, dy.
\]
In fact, this definition extends to functions in $L^{p}(\rr^{m})$ and it turns out that the corresponding operator is $L^{p}$-bounded. It follows that, having fixed $1\leq p < +\infty$,  there exists a constant $C= C(m,p)>0$ such that, for every $u \in L^{p}(\rr^{m})$,
\begin{equation}\label{LpRieszRm}
\| \R(u) \|_{L^{p}} \leq C \| u \|_{L^{p}}.
\end{equation}
\smallskip

Let $\R_{j}$ denote the $j^{th}$ component of $\R$, which is defined via $\frac{\partial K(y)}{\partial y^{j}} $. Then, on any $\vp \in C^{\infty}_{c}(\rr^{m})$, the following  crucial relation holds:
\begin{equation}\label{RieszHessianRm}
\frac{\partial^{2} \vp}{\partial x^{i} \partial x^{j}} = \R_{i}(\R_{j} (\Delta \vp)).
\end{equation}
Whence, using \eqref{LpRieszRm}, we immediately obtain the desired Euclidean Calder\'on-Zygmund inequality
\[
\left\| \frac{\partial^{2} \vp}{\partial x^{i}\partial x^{j}} \right\|_{L^{p}} \leq C^{2} \| \Delta \vp \|_{L^{p}} 
\]
for every $\vp \in C^{\infty}_{c}(\rr^{m})$. A natural question now arises:
\begin{problem}
 What we expect to survive in the general setting of complete Riemannian manifolds?
\end{problem}
There is a stream of deep works concerning the $L^{p}$-boundedness of the Riesz transform on functions, starting from the seminal papers by Strichartz, \cite{Str-JFA}, where the question on manifolds was first proposed, and by D. Bakry, \cite{Ba-SP}, where a first answer in terms of Ricci lower bounds was given. In particular, it is known that $\R$ is bounded on the whole $L^{p}$-scale, $1<p<+\infty$, if the Ricci curvature is nonnegative but intriguing (topological) obstructions for some values of $p$ appear as soon as we relax $\ric \geq 0$ to $\ric \geq -K$. In this case, it is still proved by Bakry in \cite{Ba-SP} that the $L^{p}$-boundedness can be obtained for the {\it shifted Riesz transform}\footnote{sometimes called also {\it local} Riesz transform.}  $(-\Delta + a)^{-\frac{1}{2}}$, with $a>0$. Thus, for instance, if $(M,g)$ is Ricci lower bounded and has a spectral gap $\lambda_{1}(-\Delta)>0$ (as in the Hyperbolic space), we  recover the validity of \eqref{LpRieszRm}.

On the other hand, it is clear that \eqref{RieszHessianRm} must be replaced by something defined co-variantly on the underlying manifold. In fact, recall that $\Hess(u) = Ddu$, where $D$ is the covariant derivative on $\Lambda^{1}(M)$ induced by the Levi-Civita connection of $M$. Therefore, in order to implement a version of \eqref{RieszHessianRm} that, in a way similar to the Euclidean setting, yields the validity of the Calder\'on-Zygmund inequality, we have to study {\it covariant versions of the Riesz transform}. This topic is already present in the literature. See for instance \cite{TW-POTA} which is very relevant for the next section, \cite{Lo-JFA} by N. Lohou\'e, and  also \cite{MMV} by G. Mauceri, S. Meda and M. Vallarino. However, apparently, many aspects still require a deeper understanding.\smallskip

To the best of our knowledge, the first paper where (covariant) Riesz transform techniques are used to get $L^{p}$-Calder\'on-Zygmund inequalities on Riemannian manifolds is \cite{GP-AdvMath}. Further investigations are announced in \cite{Ca-Revista}. We are going to briefly outline the main steps of the argument.

\subsection{\ref{CZ}(p), $1<p\leq 2$, via Riesz transform}
 
 Let $(M,g)$ be a complete, $m$-dimensional Riemannian manifold satisfying $\ric \geq -K$, $K >0$. We consider the operators
 \[
 d (-\Delta + K+1)^{-\frac{1}{2}} : L^{2}(M) \to L^{2}\Lambda^{1}(M),
 \]
and
 \[
D (\Delta_{1} + K+1)^{-\frac{1}{2}}: L^{2}\Lambda^{1}(M) \to L^{2}T^{2}_{0}(M),
 \]
 where $\Delta_{1} = d\delta + \delta d$ is the Hodge Laplacian on $1$-forms, $D$ denotes its covariant derivative and $T^{2}_{0}(M)$ is the vector bundle of $2$-covariant tensor fields. Since, obviously,
\[
-\Delta + K +1 = D^{\ast}D + K +1 \geq D^{\ast}D +1
\]
and, by the Weitzenb\"ock formula,
 \[
 \Delta_{1} + K +1 = D^{\ast} D + \ric + K +1 \geq D^{\ast} D +1,
 \]
we have that both these Riesz transforms are bounded in $L^{2}$ by $1$; see \cite[Lemma 4.17]{GP-AdvMath}.
\smallskip

Now, we want to prove that, for a suitable range of values of $p$, there exists a constant $C>0$ such that the inequality
\[ \tag{CZ(p)}
\| \Hess(\vp) \|_{L^{p}} \leq C \{ \| \Delta \vp\|_{L^{p}}+ \| \vp \|_{L^{p}} \}
\]
holds for every $\vp \in C^{\infty}_{c}(M)$. Clearly, it is enough to show that
\[
\| D d \vp \|_{L^{p}}  \leq C \| (-\Delta  + K+1) \vp\|_{L^{p}}, \text{ on }C^{\infty}_{c}(M).
\]
which, in turn, is equivalent to
 \[
 \|D d (- \Delta + K+1)^{-1} \vp \|_{L^{p}}  \leq C \| \vp \|_{L^{p}}.
 \]
 Using the spectral calculus we write
  \[
 (-\Delta + K+1)^{-1} = (-\Delta + K+1)^{-\frac{1}{2}}(-\Delta + K+1)^{-\frac{1}{2}}.
 \]
On the other hand, it can be deduced  from e.g. \cite[Remark B.12]{DT-JFA}, that the following commutation rule holds:
 \[
 d (-\Delta + K+1)^{-\frac{1}{2}} = (\Delta_{1} + K+1)^{-\frac{1}{2}}d.
 \]
Whence, we are reduced to prove that
\[
 \|D (\Delta_{1} + K+1)^{-\frac{1}{2}} d (-\Delta + K +1)^{\frac{1}{2}}\vp \|_{L^{p}}  \leq C \| \vp \|_{L^{p}}
 \]
and this is implied by the existence of some constants $C_{1},C_{2}>0$ such that
\begin{align}\label{Bakryestimate}
  \| d (-\Delta + K+1)^{-\frac{1}{2}} \vp \|_{L^{p}}  &\leq C_{1} \| \vp \|_{L^{p}} \\ \medskip 
  \| D (\Delta_{1} + K+1)^{-\frac{1}{2}} \w \|_{L^{p}}  &\leq C_{2} \| \w \|_{L^{p}}. \label{covariantestimate}
\end{align}
Since \eqref{Bakryestimate} holds by the fundamental work of Bakry, \cite{Ba-SP}, we have obtained the following abstract result.
\begin{proposition}\label{prop-covariantRiesz}
 Let $(M,g)$ be a complete, $m$-dimensional Riemannian manifold satisfying $\ric \geq -K$ for some $K>0$. Let also $1<p<+\infty$ be fixed. If the (shifted) covariant Riesz transform on $1$-forms $D (\Delta_{1} + K+1)^{-\frac{1}{2}}$ is bounded in $L^{p}$ then \ref{CZ}(p) holds.
\end{proposition}
Thus everything boils down to detect suitable geometric restrictions on the Riemannian manifold  $(M,g)$ in such a way that the covariant Riesz transform is $L^{p}$. One of the most general results to our disposal, valid in the range $1<p \leq 2$, is \cite[Theorem 4.1]{TW-POTA} by A. Thalmaier and F.-Y. Wang.
Accordingly, we get the following theorem that, when compared with Theorem \ref{th-localtoglobal}, has the nice feature to work for small values of $p$ without any injectivity radius condition.
 \begin{theorem}\label{th-functanalytic}
  Let $(M^{m},g)$ be a complete, $m$-dimensional Riemannian manifold satisfying
\begin{itemize}
 \item [($a$)] $\| \riem \|_{L^{\infty}} + \|D \riem\|_{L^{\infty}}<+\infty$.
 \item [($b$)] $|B_{tr}(x)| \leq \g t^{\g}e^{t^{\d}+r^{\d}} |B_{r}(x)|$,  $\forall x\in M$, $\forall t \geq 1$, some $\g>0$ and $0 \leq \d <2$.
\end{itemize}
Then, for every $1< p \leq 2$, there exists a constant $C>0$ depending on $p, m$ and the geometric data such that \ref{CZ}(p) holds for every $\vp \in C^{\infty}_{c}(M)$.
\end{theorem}
\begin{remark}
Due to the restriction $\d<2$,  the {\it generalized volume doubling condition} (b) excludes hyperbolic geometries.
\end{remark}
Beside the final aspect of Theorem \ref{th-functanalytic}, that comes from a black-box application of Thalmaier-Wang result, what we think is really relevant in this Section is that, thanks to the connection with $L^{p}$-Calder\'on-Zygmund inequalities, we have a new perspective in the study of (both shifted and genuine) covariant Riesz transforms.

\newpage

\section{From $C^{\infty}_{c}$ to $L^{p}$: special cut-off functions}

In Sections \ref{section-localtoglobal} and \ref{section-functionalanalytic} we have encountered two different methods to prove the validity of a-priori $L^{p}$-Hessian estimates on $C^{\infty}_{c}$ functions. They involve various geometric conditions on the underlying complete manifold. Following the scheme we have outlined in the $L^{2}$ case, it is now the time to extend the results to functions in the space $C^{\infty}\cap\tW^{2,p}$. Namely, the goal is to prove that:\smallskip

\noindent {\it if we are given a solution $u \in C^{\infty}(M)$ of the Poisson equation $\Delta u = f$ with $u,f \in L^{p}(M)$ then $|\Hess(u)| \in L^{p}(M)$ and, in fact, $\| \Hess(u) \|_{L^{p}}\leq C\{ \| u \|_{L^{p}} + \| f \|_{L^{p}}\}$, for some constant $C>0$ independent of $u$.}\smallskip

\subsection{Density via second order cut-offs}
The most natural way to switch from $C^{\infty}_{c}$ to $L^{p}$ integral estimates is to use density arguments. These, in turn, take advantage from the existence of a special family of cut-offs. Actually, in the very special case of Calder\'on-Zygmund inequalities, one can also use a somewhat different argument, as explained in Section \ref{section-densitywithoutcutoff}. But the approach via cut-offs is so important, general and flexible in terms both of the inequalities involved and in the underlying geometric conditions, that deserve to be dealt with in detail.

In view of our purposes, following \cite{GP-AdvMath,IRV-IMRN}, we set the next
\begin{definition}\label{def-Hessiancutoff}
 Say that $\vp_{k} \in C^{\infty}_{c}(M)$ is a sequence of (weak\footnote{the term ``weak'' refers to the fact that we do not require $ \| \Hess(\vp_{k}) \|_{L^{\infty}}\to 0$. Whenever this decaying condition is satisfied we speak of (genuine) Hessian cut-off functions.})  Hessian cut-off functions if the following conditions are met:
\begin{itemize}
 \item [(a)] $\vp_{k} \to 1$, as $k \to +\infty$.\smallskip
 \item [(b)] $\| \nabla \vp_{k} \|_{L^{\infty}} \to 0$, as $k \to +\infty$. \smallskip
 \item [(c)] $\| \Hess (\vp_{k}) \|_{L^{\infty}} \leq C$, for some constant $C>0$.
\end{itemize}
\end{definition}

In general, on a complete Riemannian manifold, such a sequence does not exist. Indeed, for instance, its presence forces the Sobolev density $W^{2,p}(M) = \overline{C^{\infty}_{c}(M)}$ but, according to Veronelli example, \cite{Ve-Counterexample}, this is not always the case.\smallskip

Among all the possible cut-off functions one can construct, the most common in Geometric Analysis are those with radial symmetry. The reason is that their properties, in terms of control of the derivatives, can be read directly in the geometry (say curvature restrictions) of the space. Here is a suggestive example.

\begin{example}
Let $(M,g)$ be a complete, $m$-dimensional Riemannian manifold satisfying  $\| \sect \|_{L^{\infty}} <+\infty$. Assume also that $r_{\mathrm{inj}}(o)=+\infty$ for some fixed origin $o \in M$ (e.g. the pole of a model manifold $\mm^{m}_{\s})$. Set $d(x) = \dist(x,o)$. Then:
\begin{itemize}
 \item [(i)] $d$ is smooth on $M\setminus \{o\}$ and proper.
 \item [(ii)] $| \nabla d| = 1$ by the Gauss Lemma.
 \item [(iii)] $| \Hess (d) | \leq C$ on $M\setminus B_{1}(o)$ by the Hessian comparison theorem, \cite{PRS-book}.\smallskip
\end{itemize}
Clearly, the singularity of $d$ can be smoothed out without touching all its good properties. For instance, take $h : \rr_{\geq 0} \to \rr_{\geq 0}$ such that
\begin{itemize}
 \item [(j)] $h(s)$ is smooth (and even convex if we like).
 \item [(jj)] $h(s) = s^{2}$ for $0 \leq s \ll 1$ and  $h(s) =As+B$, $s \gg 1$,
\end{itemize}
and define $\tilde d: M \to \rr_{\geq 0}$ by
\[
\tilde d(x) = h(d(x)).
\]
Then, $\tilde d$ still satisfies (i)--(iii) on $M$. Finally, take any smooth function $\vp: \rr \to [0,1]$ s.t. $\vp(t) = 1$ on $|t|<1$ and $\vp(t) = 0$ on $|t| >2$ and, for every $k \in \nn$, let
\begin{equation}\label{hessiancutoff}
\vp_{k}(x) = \vp \left( \frac{\tilde d(x)}{k} \right).
\end{equation}
Then, $\{ \vp_{k}\}$ s a sequence of genuine Hessian cut-off functions.

In the same vain, one can relax the curvature condition to $|\sect|(x) \leq C (1 + d(x)^{2})$ and, correspondingly, replace (iii) by
\begin{itemize}
 \item [(iii')] $|\Hess (d) |(x) \leq C'd(x)$ on $M \setminus B_{1}(o)$.
\end{itemize}
This is enough to get, via \eqref{hessiancutoff}, the desired sequence of weak Hessian cut-offs.
\end{example} 
In the previous example, if we remove the injectivity radius condition,  things became much more complicated. Although the philosophy underlying the construction of Hessian cut-offs is the same, we now have to take care of the fact that $d(x)$ is only Lipschitz and the smoothing procedure requires extra nontrivial work. As far as we know, the problem was first considered in \cite{CG-JDG} by. J. Cheeger and M. Gromov where they assume that the sectional curvature is bounded and use the mollifiers technique\footnote{one of the deep insights of their proof is that, whenever you have a sectional curvature bound and are interested in $L^{\infty}$ estimates then, by lifting locally to the tangent space via the exponential map, no injectivity radius assumption is needed.}. The proof in  \cite{CG-JDG} is rather sketchy but, using a completely different argument, based on heat kernel methods, L.-F. Tam, \cite{Ta-ALM}, provided a complete proof. Actually, M. Rimoldi and G. Veronelli, \cite{RV-CalcVar}, observed  that Tam proof works as well by assuming that the Ricci tensor is bounded and the injectivity radius has a positive lower bound. Only very recently it was realized that {\it bounded geometry} is a too much strong restriction  as the following result by Impera-Rimoldi-Veronelli shows; \cite{IRV-IMRN}.

\begin{theorem}
 Let $(M,g)$ be a complete Riemannian manifold. Having fixed a reference origin $o \in M$, let $d(x) = \dist(x,o)$. If, for some $0 \leq \eta \leq 1$, either
\begin{itemize}
 \item [($a_{1}$)] $|\ric|(x) \leq C (1+d(x)^{2})^{\eta}$
 \item [($a_{2}$)] $r_{\mathrm{inj}}(x) \geq C (1+d(x))^{-\eta}$
\end{itemize}
or
\begin{itemize}
 \item [($b$)] $| \sect |(x) \leq C (1+d(x)^{2})^{\eta}$
\end{itemize}
for some constant $C>0$, then, there exists a (distance-like) function $\tilde d \in C^{\infty}(M)$ such that
\[
\| d - \tilde d \|_{L^{\infty}} < +\infty, \quad \| \nabla \tilde d \|_{L^{\infty}} < +\infty, \quad \left\| (1+d)^{-\eta}| \Hess(\tilde d) |\right\|_{L^{\infty}}< +\infty.
\]
\end{theorem}

Summarizing: in any of the sets of assumptions in Theorems \ref{th-localtoglobal} and \ref{th-functanalytic} we have the existence of a sequence of (genuine) Hessian cut-off functions $\{ \vp_{k}\}$. Therefore, if we are given a solution $u \in C^{\infty}(M)$ of the Poisson equation $\Delta u = f$, with $u,f \in L^{p}(M)$, evaluating \ref{CZ}(p) along $u \vp_{k}$ we get
\[
\| \vp_{k} \Hess(u) \|_{L^{p}}  \leq C \| u \vp_{k}\|_{L^{p}}  + C \| |\nabla \vp_{j}| |\nabla u|\|_{L^{p}}  + C \| \vp_{k} f\|_{L^{p}}  + C \| u \Hess(\vp_{k}) \|_{L^{p}} .
\]
and taking the limits as $k \to +\infty$ gives the desired estimate
\[
 \| \Hess(u) \|_{L^{p}}  \leq C \{ \|u\|_{L^{p}}  + \| f \|_{L^{p}} \}.
\]

\subsection{Density via Calder\'on-Zygmund}\label{section-densitywithoutcutoff} Let, again, $1 < p < +\infty$. If one is interested only in Calder\'on-Zygmund inequalities then there is a short-cut to get the extension from $C^{\infty}_{c}$ to $L^{p}$. In fact, we have observed in Proposition \ref{prop-density} that, if $(M,g)$ is a complete Riemannian manifold satisfying \ref{CZ}(p) on $C^{\infty}_{c}(M)$ then  the following chain of inequalities of Sobolev spaces holds true:
\[
W^{2,p}(M) = W^{2,p}_{0}(M) =\tW^{2,p}_{0}(M) = \tW^{2,p}(M).
\]
Thus, given $u \in C^{\infty}(M)$ satisfying $u \in L^{p}(M)$ and $\Delta u \in L^{p}(M)$, we have:
\begin{itemize}
 \item $|\nabla u | \in L^{p}(M)$, $|\Hess(u) | \in L^{p}(M)$, so that, $u \in C^{\infty}(M) \cap W^{2,p}(M)$, and
 \item there exists a sequence $\vp_{k} \in C^{\infty}_{c}(M)$ such that $\vp_{k} \to u$ in $W^{2,p}(M)$.
\end{itemize}
As a consequence, by applying \ref{CZ}(p) to $\vp_{k}$, gives
\[
\| \Hess(\vp_{k}) \|_{L^{p}} \leq C \{ \| \vp_{k} \|_{L^{p}} + \| \Delta \vp_{k} \|_{L^{p}} \}
\]
and by taking the limit as $k \to +\infty$ we conclude that this inequality extends to $u$ (with the same constant)
\[
\| \Hess(u) \|_{L^{p}} \leq C \{ \| u \|_{L^{p}} + \| \Delta u \|_{L^{p}} \}
\]
as desired.

\section{$L^{p}$-gradient estimates}\label{section-gradient}

In order to complete the picture on the global $W^{2,p}$-estimates for smooth solutions of the Poisson equation, we need to investigate the  validity of $L^{p}$-estimates of the gradient.

\subsection{The case $1<p<+\infty$}\label{section-Lpgradient}
We saw in Corollary \ref{cor-grad-compactlysupp} that the inequality
\begin{equation}\label{gradestcomp}
\|  \nabla \vp \|_{L^{p}} \leq  C \cdot \left\{  \| \vp \|_{L^{p}} + \|
\Delta  \vp  \|_{L^{p}}  \right\}, \, \forall \vp \in C^{\infty}_{c}(M)
\end{equation}
holds, for some universal constant $C>0$, in the following situations:
\begin{itemize}
 \item for $1<p \leq 2$ on any complete manifold, thanks to  the interpolation inequality \eqref{interpolation2};
 \item for $p \geq 2$ on a (possibly incomplete) manifold that supports \ref{CZ}(p), thanks to interpolation inequality \eqref{interpolation1}.
\end{itemize}
On the other hand, we have already mentioned the Milatovic density result, \cite[Appendix A]{GP-AMPA}, stating that,
\[
(M,g) \text{ complete } \Rightarrow \tW^{2,p}(M) = \overline{C^{\infty}_{c}(M)}^{\tW^{2,p}}
\]
in the range $1<p \leq 2$\footnote{remember that $p=2$ is Strichartz seminal observation, \cite{Str-JFA}.}. Whence, as we did in the $L^{2}$ setting, we immediately get from \eqref{gradestcomp} the following:

\begin{theorem}\label{th-gradient-milatovich}
 Let $(M,g)$ be a complete Riemannian manifold. Fix $1<p \leq 2$. Then, there exists a constant $C>0$ such that, if $u \in C^{\infty}(M)$ is a solution of the Poisson equation $\Delta u = f$ with $u,f \in L^{p}(M)$ then
 \[
 \|  \nabla u \|_{L^{p}} \leq  C \cdot \left\{  \| u \|_{L^{p}} + \| f  \|_{L^{p}}  \right\}.
 \]
\end{theorem}
The case $2<p<+\infty$ is more delicate and, apparently(!), requires some restriction on the geometry. To what extent the geometry must be controlled is not completely clear due the lack of concrete counterexamples. To the best of our knowledge, the most general result so far known is the following theorem by L. J. Cheng, A. Thalmaier and J. Thompson, \cite{CTT-JMAA}. Its  proof is based on stochastic calculus.

\begin{theorem}\label{th-CTT}
 Let $(M,g)$ be a complete, $m$-dimensional Riemannian manifold satisfying $\ric \geq - K^{2}$, for some $K \geq 0$. Fix any $1<p<+\infty$. Then, there exists a constant $C=C(m,p,K)>0$ such that the conclusion of Theorem \ref{th-gradient-milatovich} holds true.
\end{theorem}

We are going to provide a different (and somewhat direct) argument, based on Riesz transform estimates and deep facts in operator theory. It was suggested to me by Stefano Meda.

\begin{proof}
Let $u \in \tW^{2,p}(M)$, i.e., $u,\Delta u \in L^{p}(M)$. By Milatovich density result, Theorem \ref{th-Milatovich}, there exists an approximating sequence $\{ \vp_{j}\} \subset C^{\infty}_{c}(M)$ such that
\[
i) \, u = \lim_{L^{p}} \vp_{j}, \quad ii) \, \Delta u = \lim_{L^{p}} \Delta \vp_{j}.
\]
Now we recall that, under Ricci lower bounds, we have the Bakry $L^{p}$ estimate \eqref{Bakryestimate} of the shifted Riesz transform. On the other hand (see e.g. \cite{Ba-SP})
\[
\| (-\Delta + K+1)^{\frac{1}{2}}\vp_{j} \|_{L^{p}} \leq C \left\{ \sqrt{K+1} \| \vp_{j} \|_{L^{p}} + \| (-\Delta)^{\frac{1}{2}} \vp_{j} \|_{L^{p}} \right\}.
\]
Therefore
\[
\|  \nabla \vp_{j} \|_{L^{p}} \leq  C'  \left\{  \| \vp_{j} \|_{L^{p}} + \| (-\Delta)^{\frac{1}{2}}  \vp_{j}  \|_{L^{p}}  \right\}.
\]
It remains to take care of the last summand. Here, since the Laplace-Beltrami operator is {\it sectorial}, we can appeal to what is known in the literature as the {\it moment inequality}, see \cite[Proposition 6.6.4]{Haa} according to which
\[
\| (-\Delta )^{\frac{1}{2}}\vp_{j} \|_{L^{p}} \leq C''\,  \| \vp_{j} \|_{L^{p}}^{\frac{1}{2}} \,  \| \Delta \vp_{j} \|_{L^{p}}^{\frac{1}{2}}
\]
holds with some uniform constant $C''>0$. Summarizing,
\[
\|  \nabla \vp_{j} \|_{L^{p}} \leq  C''  \left\{  \| \vp_{j} \|_{L^{p}} + \| \Delta  \vp_{j}  \|_{L^{p}}  \right\}.
\]
for some constant $C= C(m,p,K)>0$. This estimate implies
\[
\nabla u = \lim_{L^{p}}\nabla \vp_{j}
\]
and, thus, it extends to $j \to +\infty$. The proof is completed.
\end{proof}

\begin{remark}
 Tracing back the dependence of the constant $C$ on the parameters, and precisely looking at the paper by Bakry, \cite{Ba-SP}, we see that $C$ depends linearly on the lower curvature bound $K$. This recovers (once optimized with respect to the parameter involved) what is obtained in \cite{CTT-JMAA} using stochastic methods.
\end{remark}

We conclude this section with some abstract considerations. As we have already remarked in the Introduction, gradient estimates are important in themselves. However, if our interest is merely in the Calder\'on-Zygmund theory, a funny phenomenon appears. Namely: for $2<p<+\infty$, and on a complete manifold $(M,g)$, the implication
\begin{equation}\label{gradient-vs-hessian}
\begin{cases}
\Delta u = f \\
u,f \in C^{\infty}(M) \cap L^{p}(M) \\ 
\end{cases}
\Rightarrow \, |\nabla u | \in L^{p}
\end{equation}
is just a formal consequence of the Calder\'on-Zygmund inequality. To see this, following \cite{GP-AMPA},  observe that, for any function $u \in C^{\infty}(M)$, the inequality
\[
\| \nabla u \|_{L^{p}}^{p} \leq \| u \Delta_{p}u \|_{L^{1}}
\]
can be obtained using integration by parts and $1^{st}$-order cut-off functions. Here we have adopted the notation $\Delta_{p} u = \div(|\nabla u|^{p-2}\nabla u)$  for the  {\it $p$-Laplacian} of $u$. Whence, unwinding  the definition of $\Delta_{p}$ and using Young and H\"older inequalities we obtain that, for a solution $u \in L^{p}$ of $\Delta u  = f \in L^{p}$,
\begin{equation}\label{interpolation-hessian}
 \| \nabla u \|_{L^{p}}^{2} \leq C \| u \|_{L^{p}} \left(  \| f \|_{L^{p}} + (p-2)\| \Hess(u) \|_{L^{p}} \right) <+\infty.
\end{equation}
The claimed implication \eqref{gradient-vs-hessian} now follows from \eqref{interpolation-hessian} provided that the $L^{p}$-Hessian estimate holds on $M$. We can summarize what we have seen in the following
\begin{proposition}\label{prop-gradient-Hessian}
 Let $(M,g)$ be a complete Riemannian manifold and let $2 \leq p <+\infty$. Assume that there exists a constant $C>0$ such that
 \[
 \| \Hess(u) \|_{L^{p}} \leq C\{ \| u \|_{L^{p}} + \|\Delta u \|_{L^{p}}\},\, \forall  u \in \tW^{2,p}(M).
 \]
Then, there exists a new constant $C'>0$ such that
 \begin{equation}\label{Lpgradient-Baptiste}
 \|  \nabla u \|_{L^{p}} \leq  C' \cdot \left\{  \| u \|_{L^{p}} + \| \Delta u \|_{L^{p}}  \right\}, \, \forall  u \in \tW^{2,p}(M).
 \end{equation}
\end{proposition}
The moral of this abstract observation is that, if we are searching for an $L^{p}$-Hessian estimate on a complete Riemannian manifold, with $p \geq 2$, then, whatever the geometric restrictions on the manifold are, the validity of an $L^{p}$-gradient estimate is a necessary condition. This fact could be used to find new examples where the $L^{p}$-Hessian estimate does not hold. ``Simply'' construct a complete manifold where the $L^{p}$-gradient estimate is not satisfied.

A closely related and interesting question is how much the $L^{p}$ gradient estimates are subordinated to the validity of Calder\'on-Zygmund inequalities. In this respect, it was asked by Baptiste Devyver whether it is possible to construct a (smooth!) complete Riemannian manifold for which \eqref{Lpgradient-Baptiste} holds true but the corresponding \ref{CZ}(p) inequality (on compactly supported functions) does not hold.

\subsection{The endpoint case $p=+\infty$}
To complete the picture on the $L^{p}$-gradient estimates we have to discuss what happens in the (only admissible) endpoint case $p=+\infty$. This is not related to any Calder\'on-Zygmund theory but it is an important topic (even more important) that permeates the whole Geometric Analysis. The ``founding fathers'' in the context of manifolds with lower Ricci bounds\footnote{in the Euclidean space, gradient estimates for bounded solutions of the Poisson equation $\Delta u =f(u)$ were known to L. Modica, \cite{Mo-CPAM}.} are S.Y. Cheng and S.T. Yau, \cite{CY} with their ubiquitous {\it local gradient estimates} for positive harmonic functions. Accordingly if $B_{2R}(o)$ is a relatively compact ball in the $m$-dimensional Riemannian manifold $(M,g)$ and $\ric \geq -(m-1)K$ on $B_{2R}(o)$ for some constant $K \geq 0$, then every  harmonic function $u>0$ on $B_{2R}(o)$ satisfies
\[
\sup_{B_{R}(o)} | \nabla \log u |(x) \leq (m-1) \sqrt{K} + \frac{C}{R},
\]
where $C>0$ is a dimensional constant.  Actually, one can enlarge the class of equations to $\Delta u = a(x) g(u)$ and the Ricci tensor is also allowed to decay to $-\infty$ provided the asymptotic behaviour of $a(x)$ and the Ricci decay are suitably related to each others. This is a contribution of Bianchi-Setti, \cite{BS-CalcVar}. When the sign restriction is removed, gradient estimates for bounded solutions of $\Delta u = g(u)$ were previously obtained in \cite{RR-Tohoku} by A. Ratto and M. Rigoli. All these papers have in common the special structure of the datum $f$ in the Poisson equation $\Delta u = f$. But in order to merely get a gradient bound in terms of the solution itself, this restriction can be removed. The following result, stated for the drifted Laplacian and with upper integral bounds, can be found in \cite[Theorem 4.1]{ZZ-AdvMath} by Q. Zang and M. Zhu. See also P. Li book \cite{Li-book}.

\begin{theorem}
 Let $(M,g)$ be a complete $m$-dimensional Riemannian manifold satisfying $\ric \geq -K$, for some $K \geq 0$. Then, there exist a constant $C=C(m,K)>0$ and a radius $r_{0}=r_{0}(m,K)>0$ such that, for all balls $B_{r}(o)$ with $0<r\leq r_{0}$ and  $o \in M$, and for all solutions $u \in C^{\infty}(M)$ of
 \[
 \Delta u = f, \, \text{on }B_{r}(o)
 \]
it holds
\[
 \|\nabla u \|_{L^{\infty}(B_{r/2}(o))} \leq C \{ r^{-2} \| u \|_{L^{\infty}(B_{r}(o))} + \| f \|_{L^{\infty}(B_{r}(o))} \}.
\]
In particular, if the equation is satisfied on all of $M$, then we obtain the global $L^{\infty}$-gradient estimate
\[
 \|\nabla u \|_{L^{\infty}} \leq C \{ \| u \|_{L^{\infty}} + \| f \|_{L^{\infty}}\}.
\]
\end{theorem}
The proof relies on Moser iteration and the constant has a somewhat implicit dependence from the curvature bound. Using a completely different argument, everything can be quantify if we replace the Ricci assumption with a double-side control on the sectional curvature. In fact, we have the following result from \cite{GP-IMRN}.
\begin{theorem}
 Let $(M,g)$ be a complete Riemannian manifold, $m = \dim M$. Then there exists a dimensional constant $C=C(m)>0$ such that, for all $x \in M$ and $r>0$ such that the following holds.\smallskip
 
 \noindent If $u$ is a smooth solution of $\Delta u = f$ in  $B_{r}(x)$ then, for any $\e>0$,
 \[
 \| \nabla u \|_{L^{\infty}(B_{r/4}(x))} \leq \frac{C}{\min \left(1,r, (\| \sect \|_{L^{\infty}(B_{r/2}(x))}+ \e)^{-1/2} \right)}  \left\{ \| f \|_{L^{\infty}(B_{r/2}(x))} + \| u \|_{L^{\infty}(B_{r/2}(x))}\right\}.
 \]
\end{theorem}

\section{$L^{p}$-gradient estimates: counterexamples?}

As the question mark in the title of the Section suggests, the picture concerning  the failure of the $L^{p}$-estimates of the gradient (and consequently also of their validity) at the moment is only sketchy (to me, obviously).\smallskip

A first natural question concerns the range of values of $p$ where it is reasonable to consider these inequalities. According to Remark \ref{rem-cz-endpoint}, in the study of the validity of $L^{p}$-Hessian estimates on a complete Riemannian manifold  we have to exclude the endpoint cases $p=1$ and $p=+\infty$. Thus, if we are mainly interested in developing a Calder\'on-Zygmund theory, then the $L^{p}$-gradient estimates are just companion of the Hessian estimates, and it is natural as well to maintain the range $1<p<+\infty$. On the other hand, the interest in the gradient estimates is independent of any Calder\'on-Zygmund theory and permeates the Geometric Analysis. Since, from the very beginning of this survey, we have decided to separate the study of first and second order inequalities we still have to discuss what happens if $p=1$ or $p=+\infty$. Now, we shall see in Section \ref{section-gradient} that a lower Ricci bound is enough to get $L^{\infty}$ estimates and that, furthermore, we can improve the dependence of the constant on the geometry when we have a double sided control on the sectional curvature. The case $p=1$ looks drastically different. The corresponding estimate should fail even in the Euclidean space as one can try to prove using Hardy space theory; \cite{M}.
\\

In the rest of the section we shall restrict our attention to the range $1<p<+\infty$. At the present state of  understanding, and as we have discussed in Sections \ref{section-interpolation} and \ref{section-gradient}, in order to violate the $L^{p}$-gradient estimates we have the following few chances: (a) if $1<p \leq 2$, we have to use a geodesically incomplete manifold; (b) if $2< p <+\infty$ then: (b.1) we can  focus on manifolds that do not support either (\ref{CZ})(p) or a sequence of cut-off functions with controlled Hessian; (b.2) otherwise we can use a complete manifold with Ricci curvature unbounded from below. To be honest, concerning this second alternative, it is not at all clear to what extent a controlled decay at $-\infty$ is really forbidden.\smallskip

From the viewpoint of $L^{p}$ functions, the trivial situation is that represented by (a). For instance, the punctured Euclidean space $M = \rr^{m}\setminus \{ 0\}$, $m \geq 3$, endowed with its flat metric does not support an $L^{p}$-gradient estimate for every $\frac{m}{m-1}< p < \frac{m}{m-2}$.\smallskip

To see this, let us consider the Green function of $\rr^{m}$ with pole at $0$:
\[
G(x) = \frac{1}{r(x)^{m-2}}.
\]
Since $G$ is harmonic, the $L^{p}$-norm of its Laplacian is trivially finite. Moreover, since
\[
\int_{0^{+}} \frac{r^{m-1}}{r^{p(m-2)}} dr <+\infty,
\]
integrating in polar coordinates shows that the singularity of $G$ is $L^{p}$-integrable. On the other hand,
\[
|\nabla G(x)|^{p} = \frac{ (m-2)^{p}}{r(x)^{p(m-1)}}
\]
with
\[
\int_{0^{+}}\frac{r^{m-1}}{r^{p(m-1)}} dr = \int_{0^{+}} \frac{dr}{r^{(m-1)(p-1)}} = +\infty
\]
showing that the singularity of $|\nabla G|$ is not $L^{p}$-integrable. Thus, having fixed any cut-off function $\vp \in C^{\infty}_{c}(\rr^{m})$ satisfying $\vp = 1$ near $0$ we get that the function $u = \vp\, G \in C^{\infty}(M)$ violates the $L^{p}$-gradient estimate, as expected.

Note  that there is no hope to violate the $L^{p}$-gradient estimate at the $C^{\infty}_{c}$ level. Indeed, $M$ is an open set of $\rr^{m}$ where the $L^{p}$-gradient estimate on $C^{\infty}_{c}$-functions holds for any $1<p<+\infty$. Note also that the choice of the flat space $\rr^{m}$ is  inessential because, given a Riemannian manifold $(M,g)$, in a small neighborhood of a reference point $\bx \in M$, the metric has the polar expression $g = dr \otimes dr + (r^{2} \d_{ij}+ o(r^{2})) \theta^{i}\otimes\theta^{j}$ where $g_{\ss^{m-1}}=\sum \theta^{i} \otimes \theta^{i}$ is the standard metric of $\ss^{m-1}$. Moreover, near the singularity, the Green kernel with pole at $\bx$ and its gradient have exactly the same behaviour as the Euclidean one.

\section{Final result and an application}

Summarizing what we have seen in the previous sections, we have the following comprehensive (although far from being complete) picture.

\begin{theorem}[global $W^{2,p}$ regularity]\label{th-globalW2p}
Let $(M,g)$ be a complete, $m$-dimensional Riemannian manifold and let $p \in \rr$. Assume that one of the following sets  of assumptions ($\A$), ($\B$), ($\C$) or ($\D$) is satisfied:
\begin{itemize}
 \item [($\A_{1}$)] $\ric \geq -K$ for some $K \geq 0$;
 \item [($\A_{2}$)] $p=2$
\end{itemize}
or 
\begin{itemize}
 \item [($\B_{1}$)] $\|\ric \|_{L^{\infty}}<+\infty$;
 \item [($\B_{2}$)] $r_{\mathrm{inj}}(M)>0$;
 \item [($\B_{3})$] $p \in (1,+\infty)$;
\end{itemize}
or
\begin{itemize}
 \item [($\C_{1}$)] $\| \riem \|_{L^{\infty}} + \|D \riem\|_{L^{\infty}}<+\infty$;
 \item [($\C_{2}$)] $|B_{tr}(x)| \leq \g t^{\g}e^{t^{\d}+r^{\d}} |B_{r}(x)|$,  $\forall x\in M$, $\forall t \geq 1$, some $\g>0$ and $0 \leq \d <2$;
 \item [$(\C_{3})$] $p \in (1,2]$;
\end{itemize}
or
\begin{itemize}
 \item [($\D_{1}$)] $\| \riem \|_{L^{\infty}} <+\infty$;
 \item [($\D_{2}$)]  $p \in [2,+\infty) \cap (m/2,+\infty)$.
\end{itemize}
Then, there exists a constant $C = C(m,p,\mathrm{Geom(M)})>0$ such that the following holds.\smallskip

\noindent Let $u \in C^{\infty}(M)$ be a solution of the Poisson equation:
\[
\Delta u = f, \quad \text{on }M.
\]
If $u, f \in L^{p}(M)$ then $u \in W^{2,p}(M)$ and
\begin{equation}\label{W2pestimate}
\| \nabla u \|_{L^{p}} + \| \Hess(u) \|_{L^{p}} \leq C \left\{ \| u \|_{L^{p}} +  \| f \|_{L^{p}} \right\}.
\end{equation}
\end{theorem}

In conclusion of this survey we would like to mention a concrete situation where this result applies.
\smallskip

For any given  constant $c \leq 0$, let  $\mm^{m+1}_{\s_{c}}$ denote the spaceform (model manifold) of constant sectional curvature $\sect \equiv c$. Thus, $\mm^{m+1}_{\s_{0}} = \rr^{m+1}$ and $
\mm^{m+1}_{\s_{-1}} = \hh^{m+1}$.

We consider an oriented, isometric immersion $f : M \to \mm^{m+1}_{\s_{c}}$ of the complete, $m$-dimensional Riemannian manifold $(M,g)$ into $\mm^{m+1}_{\s_{c}}$. Its Gauss map is denoted by $\nu$. A widely studied family of such hypersurfaces are those with  constant mean curvature (CMC for short) $\bH (x) = H \nu(x)$, $H \in \rr$. By an isoperimetry argument, the volume of unit balls of $M$ does not collapse at infinity, namely:
\[
\inf_{x \in M} \vol B_{1}(x) = v>0.
\]
Thus, if we also assume that the hypersurface has bounded second fundamental form
\[
| \bA | \in L^{\infty}(M)
\]
then, by Gauss equations,
\[
\| \sect \|_{L^{\infty}} = S <+\infty.
\]
and it follows from \cite{CGT-JDG} that also the injectivity radius is lower bounded by a positive constant:
\[
r_{\inj}(M) = i >0.
\]
Therefore, Theorem \ref{th-globalW2p} applies and gives that, for any $1<p<+\infty$, there exists a constant $C=C(m,p,S,i)>0$ such that if $u \in C^{\infty}(M)$ is a solution of the Poisson equation $\Delta u = f$ with $f,u \in L^{p}(M)$ then the global $W^{2,p}$ estimate \eqref{W2pestimate} is satisfied.

Now, interesting information on the extrinsic geometry of the CMC hypersurface $f: M \to \mm^{m+1}_{\s_{c}}$ is encoded in the kernel of its {\it stability operator}, which is the Schr\"odinger operator
\[
L = \Delta  +( | {\mathbf A}|^{2}+m c).
\]
The solutions $u \in C^{\infty}(M)$ of the corresponding equation
\[
Lu =0
\]
give rise to the vector space $\CJ(M)$ of the {\it Jacobi functions}\footnote{sometimes called {\it Jacobi fields} but this terminology may be confusing.} of the hypersurface. From the previous discussion we deduce the validity of the following result.
\begin{corollary}
 Let $f : M \to \mm^{m+1}_{\s_{c}}$ be a complete, CMC hypersurface with bounded second fundamental form in the spaceform $\mm^{m+1}_{\s_{c}}$ of constant curvature $c \leq 0$. Let $u \in \CJ(M)$ be a Jacobi function. If $u \in L^{p}(M)$ for some $p \in (1,+\infty)$ then
 \[
 \| \nabla u \|_{L^{p}} + \| \Hess(u) \|_{L^{p}} \leq C \| u \|_{L^{p}} <+\infty.
 \]
where $C>0$ is an absolute constant independent of $u$.
\end{corollary}
When $m=2$, $\mm^{m+1}_{\s_{c}}=\rr^{3}$ and $f:M \to \rr^{3}$ is a $k$-unduloid, the dimension of the space $\CJ(M) \cap L^{2}(M)$ is studied e.g. in \cite{KKR}, by N. Korevaar, R. Kusner and J. Ratzkin, in connection with the singularities of the moduli space of CMC surfaces.  Due to the periodicity of its $k$-ends, a $k$-unduloid satisfies the assumptions of the Corollary. Another class of CMC hypersurfaces to which the Corollary applies is that of {\it finite total curvature} hypersurfaces. This means that
\[
\left \vert \bA - \frac{1}{m}\bH \right\vert^{2} \in L^{q}(M),
\] 
with $q \geq m/2$. In this setting, in order to avoid nonexistence issues, one typically asks for the validity of the compatibility condition
\[
H^{2} + c \leq 0.
\]
Thus, for instance, in the Euclidean space $\mm^{m+1}_{\s_{c}} = \rr^{m+1}$, this condition forces $H=0$, i.e. $f$ is a minimal immersion. Since, in the finite total curvature assumption, one has $\left \vert \bA - \frac{1}{m}\bH \right\vert(x) \to 0$ as $x \to \infty$, see e.g. \cite{PV-DJA}, we are still in the position to apply the Corollary, as claimed.

\section{A very brief account on very recent develoments}

After the first draft of this survey was posted on arXiv, many important and beautiful results on the global Calder\'on-Zygmund theory appeared.

\subsection{Counterexamples}

\begin{itemize}
\item In the paper \cite{MV}, by using a localized version of the De Philippis-Zimbron arguments, L. Marini and G. Veronelli are able to construct a smooth, complete, $m$-dimensional Riemannian manifold $(M,g)$ with $\sect >0$ on which \eqref{CZ}(p) is violated for  $p > m$. This, in particular, implies that there are complete manifolds without \eqref{CZ}(p) where the $L^{p}$-gradient estimates hold, thus answering in the affirmative the question asked by Devyver; see the discussion after Proposition \ref{prop-gradient-Hessian}.
\item In the subsequent paper \cite{HMRV}, S. Honda, L. Mari, M. Rimoldi and G. Veronelli use a clever trick to show that, actually, counterexamples to \eqref{CZ}(p) with $\sect \geq 0$ (and nontrivial topology) can be obtained on the whole scale $2<p<+\infty$.
\item In the above mentioned \cite{MV} by Marini-Veronelli, elaborating on the examples in \cite{GP-AdvMath} and \cite{Li-AGAG}, it is shown that \eqref{CZ}(2) may fail if the Ricci lower bound is replaced by $\sect \geq -\l(\dist(x,o))$, where $\l(t)>0$ is any increasing function such that $\l(t) \to +\infty$ as $t \to +\infty$. More generally, they obtain counterexamples to \eqref{CZ}(p) in any dimension $m = \dim M \geq 2$, and for every $1<p<+\infty$. 
\end{itemize}

\subsection{Positive results}
\begin{itemize}
 \item Let $1<p<2$. In the recent preprint \cite{BDG}, using a mixture of probabilistic techniques and heat kernel estimates in the spirit of \cite{CD-CPAM}, R. Baumgarth, B. Devyver and B. G\"uneysu prove that the generalized volume doubling condition ($\C_{2}$) in Theorem \ref{th-globalW2p} can be removed. As a matter of fact, they show much more than this: namely, that a $C^{1}$-bound on the Riemann tensor is enough to have the validity of covariant Riesz transform estimates on forms and, as a direct consequence, of \eqref{CZ}(p). It is a striking result  obtained by J. Cao, L.-J. Cheng and A. Thalmaier in the very recent \cite{CCT} that the validity of \eqref{CZ}(p) can be proved in the sole assumption $\ric \geq - K$.
 \item Let $2<p<+\infty$. In the above quoted \cite{CCT}, Cao-Cheng-Thalmaier are also able to prove, for the first time, the validity of \eqref{CZ}(p) in unbounded curvature settings. More precisely, by keeping $\ric \geq -K$, they assume that $|\riem|^{2}+ |D \ric|^{2} \leq H(x)$ where the function $H(x) \geq 0$ belongs to a certain {\it Kato class}. The proof involves stochastic machineries. As a consequence, they can deduce a new density result for the Sobolev space $W^{2,p}(M)$; see Proposition \ref{prop-density}.
\end{itemize}

\appendix
\addcontentsline{toc}{section}{Appendices}

\section{Estimates of the Euclidean radius}\label{appendix-euclidean}
This section aims to show that a lower estimate of the Euclidean radius cannot depend solely on an injectivity radius bound (and an upper bound of the sectional curvature).\smallskip

It is known that closed hyperbolic $m$-dimensional manifolds may have arbitrarily large injectivity radii, \cite{Fa-Tata}. Thus, given $m \in \nn_{\geq 2}$, we find  a sequence $i_{k} \to +\infty$ such that, for every $k \in \nn$ there exists a closed $m$-dimensional Riemannian manifold $M_{k}:=(M_{k},g_{k})$ with the following properties:
\begin{enumerate}
 \item $\sect_{M_{k}} \equiv -1$.
 \item $r_{\inj}(M_{k}) = i_{k}$.
\end{enumerate}
Now, for every $k \in \nn$ we scale the metric $g_{k}$ by its injectivity radius thus obtaining a new Riemannian manifold $\tM_{k} := (M_{k} ,\tg_{k} = i_{k}^{-2}g_{k})$ such that:
\begin{enumerate}
 \item $\widetilde \sect_{k} \equiv -i^{2}_{k}$
 \item $r_{\inj}(\tM_{k}) = 1$.
\end{enumerate}
Suppose now, by contradiction, that, for any $m \in \nn_{\geq 2}$ and $i  \in \rr_{>0}$, there exists a universal constant $C = C(m,i)>0$ with the following property:\smallskip

\noindent{\it for every complete Riemannian manifold $(M,g)$ of dimension $\dim M =m$ and injectivity radius $r_{\inj}(M) = i$, the Euclidean radius $r_{\euc}(M)$ of $M$ satisfies $ r_{\euc}(M) \geq C.$}\smallskip

Specifying this inequality to the sequence $\tM_{k}$ we get
\[
r_{\euc}(\tM_{k}) \geq C,
\] 
for some constant $C = C(m)>0$. This implies that, having fixed $\widetilde o_{k} \in \tM_{k}$, there exists a coordinate chart $\widetilde\varphi_{k}: \tB_{C}(\tilde o_{k}) \to \rr^{m}$ such that
\[
2^{-1}\cdot \delta_{ij} \leq (\tg_{k})_{ij} \leq 2\cdot \delta _{ij}.
\] 
Let $R = R(m) >0$ be defined by (recall that $r_{\inj}(\tM_{k})=1$)
\[
R = \min ( 1, C ).
\]
Observe that
\[
\tilde \varphi_{k} (\tB_{R/2}(\widetilde o_{k})) \subseteq \bb_{R/\sqrt{2}}(0)
\]
and, therefore,
\begin{align} \nonumber
 \vol (\tB_{R/2}(\widetilde o_{k})) &= \int_{\widetilde\varphi_{k}(\tB_{R/2}(\widetilde o_{k}))} \sqrt {\det \tg_{k}(x)} dx\\ \nonumber
 &\leq \int_{\tilde\varphi_{k}(\tB_{R/2}(\widetilde o_{k}))} 2^{m/2} dx\\ \nonumber
 &\leq \int_{\bb_{R/\sqrt{2}}(0)} 2^{m/2} dx = C_{m}. \nonumber
\end{align}
On the other hand, since $\tB_{R/2}(\widetilde o_{k})$ is isometric to the ball of the same radius in $\hh^{m}_{-i_{k}^{2}}$, we have
\[
\vol (\tB_{R/2}(\widetilde o_{k}))) = \omega_{m} \int_{0}^{R/2} \{i_{k}^{-1} \sinh(i_{k} t)\}^{m-1} dt \to +\infty.
\]
as $k \to +\infty$. Contradiction.

\section{Poisson equation on  limit spaces}\label{section-singularPoisson} 

The proof of Theorem \ref{th-DZ} by De Philippis and Zimbr\'on, showing that the Calder\'on-Zygmund constant cannot depend solely on a lower sectional curvature bound, is slightly different from the original one. Indeed, rather than harmonic function theory, it makes use of the Poisson equation  on the limit space. In the assumptions of the theorem, this limit space supports a (non-constant) solution which is approximated by smooth solutions, of corresponding Poisson equations, on the converging sequence.

Let
\[
\CM(m,D,K)=\{ (M,g) \text{ cpt Riem. manifold}: \dim M =m, \, \diam (M) \leq D, \, \ric \geq -K \}.
\]
Its (compact) closure in the measured Gromov-Hausdorff (mGH) topology is denoted by $\overline{\CM(m,D,K)}$. Observe that, by volume comparison and the uniform bound $D$ of the diameter, there exists $V=V(m,D,K)>0$ such that $\vol X \leq V$ for every $X \in \overline{\CM(m,D,K)}$.
\begin{proposition}
 Let $(M_{k},g_{k},x_{k}) \in \CM(m,D,K)$ be a pointed sequence converging in the $\mathrm{mGH}$ topology to a non collapsed limit space $(X,x_{\infty},\mu_{\infty}) \in \overline{\CM(m,D,K)}$.
Let also $2 \leq  p<+\infty$. Then, there exist non-constant functions $u_{k} \in C^{2}(M_{k})$, $g_{k} \in \Lip(M_{k})$ and $u_{\infty} \in W^{1,2}(X) \cap L^{p}(X)$, $g_{\infty} \in L^{p}(X)$ such that:

\begin{itemize}
 \item [(a)] $\avint_{M_{k}}g_{k} = 0$ and $\avint_{X} g_{\infty} = 0$;\smallskip
 
 \item  [(b)] $\Delta_{M_{k}}u_{k} = g_{k} \to g_{\infty}  =\Delta_{X} u $ in the strong $L^{p}$ (hence $L^{2}$) sense of K. Kuwae and T. Shioya, and S. Honda, \cite[Definition 1.1, Proposition 3.31]{Ho-Crelle} \cite[Definition 2.12]{Ho-Memoirs}. In particular, $\| \Delta_{M_{k}} u_{k} \|_{L^{p}} \to \| \Delta_{X} u_{\infty}\|_{L^{p}}$; \smallskip
 
 \item [(c)] $u_{k} \to u$ in the strong $W^{1,2}$ sense, as $k \to +\infty$;\smallskip
 
 \item [(d)] $\| u_{k}\|_{W^{1,p}} \leq L$ for some constant $L = L(p,m,D,K)>0$. In particular:\smallskip
 
 	\begin{itemize}
 	\item [(d.1)] $u_{k} \to u_{\infty}$ in the strong $L^{p}$ (hence $L^{2}$) sense and, hence, $\| u_{k} \|_{L^{p}} \to \|u_{\infty}\|_{L^{p}}$;\smallskip
	\item [(d.2)] $\nabla^{M_{k}} u_{k} \to \nabla^{X} u_{\infty}$ in the weak $L^{p}$ (hence $L^{2}$) sense.
	\end{itemize}

\end{itemize}
\end{proposition}

\begin{proof}
Since $\diam M_{k} \to \diam X >0$ and, for any $r>0$, $\vol B^{M_{k}}_{r}(x_{k}) \to \vol B^{X}_{r}(x_{\infty})$, then, up to choosing  $k \gg 1$, we can assume that $\diam M_{k} \geq \frac{1}{2}\diam X$ and $\vol M_{k} \geq \frac{1}{2} \vol X$. Note also that, by volume comparison, $\vol B^{M_{k}}_{r}(x_{k}) \leq \a(r)$ for some (exponential like) function depending only on $m$ and $K$.

With this preparation, we choose  $0<R< \frac{1}{4} \diam X$ such that $\a(R) \leq \frac{1}{4}\vol X$. Thus, for $k \gg 1$,
\[
\vol B^{M_{k}}_{R}(x_{k}) \leq \a(R) \leq \frac{1}{4}\vol X \leq \frac{1}{2}\vol M_{k}.
\]
Next, we define $f_{k}:M_{k}\to [0,1]$ to be a Lipschitz function satisfying
\[
i)\, \supp(f_{k}) \subseteq B^{M_{k}}_{R}(x_{k}), \quad ii) \, f_{k} =1\, \text{ on } B^{M_{k}}_{R/2}(x_{k}), \quad iii) \, \| \nabla f_{k} \|_{L^{\infty}} \leq \frac{2}{R},
\]
and we note that
\[
\avint_{M_{k}}f_{k} \leq \frac{\vol B^{M_{k}}_{R}(x_{k})}{\vol M_{k}} \leq \frac {1}{2}.
\]
Finally, we let
\[
g_{k} = f_{k} - \avint_{M_{k}} f_{k}.
\]

Clearly, by definition, $\avint_{M_{k}} g_{k} = 0$ and $ \| g_{k} \|_{L^{\infty}} \leq 1$. Moreover, $g_{k} \geq \frac{1}{2}$ on $B^{M_{k}}_{R/2}(x_{k})$ showing that, in particular, $\| g_{k} \|_{L^{\infty}} \geq \frac{1}{2}$ and  $g_{k}\not \equiv const$. Due to the zero mean condition this is in fact equivalent to the fact that $g_{k} \not\equiv 0$.

From the uniform $L^{\infty}$-bound of $\{ g_{k} \}$ and the uniform volume upper bound $V$ of $\{ M_{k} \}$ we get $\| g_{k} \|_{L^{q}} \leq V^{1/q}$ for every $q >1$ and, therefore, there exists a subsequence (still denoted by $g_{k}$) that converges to $g_{\infty} \in L^{p}(X)$ in the strong $L^{p}$ (hence $L^{2}$) sense. Indeed, by  \cite[Proposition 3.19]{Ho-Crelle}, up to passing to a subsequence, $g_{k}$ converges weakly to a function $g_{\infty} \in L^{p}(X)$. On the other hand, since, by $iii)$, $\{ g_{k}\}$ are uniformly Lipschitz, then they are asymptotically uniformly continuous in the sense of \cite[Definition 3.2]{Ho-Crelle}. It follows from \cite[Remark 3.8 and Proposition 3.32]{Ho-Crelle} that $g_{k}$ converges in the strong $L^{p}$-sense to $g_{\infty}$, as claimed. In particular, by the definition of convergence, $\avint_{X}g_{\infty} =0$ and, since $\| g_{k} \|_{L^{p}} \to \| g_{\infty}\|_{L^{p}}$ with $\| g_{k} \|_{L^{p}} \geq \frac{1}{2}V^{1/p}$, we deduce that $g_{\infty}$ is non-constant (i.e.. non-zero).

Now, let $u_{k} \in C^{2}(M_{k})$ be the (obviously non-constant) unique solution of the Poisson equation
\[
\Delta_{g_{k}} u_{k} = g_{k},\quad \text{on }M_{k},
\]
satisfying
\[
\avint_{M_{k}} u_{k} = 0.
\]
Then, by \cite[Theorem 1.1]{Ho-Memoirs}, $u_{k}$ converges in the strong $W^{1,2}$-sense to the unique solution $u_{\infty}\in  W^{1,2}(X)$ of the Poisson equation
\[
\Delta_{X} u_{\infty} = g_{\infty} \in L^{p},\quad \text{on }X
\]
satisfying $\avint_{X} u_{\infty} = 0$.

We claim that $\{ u_{k}\}$ is a bounded sequence in $W^{1,p}(M_{k})$ and, in particular, up to passing to a subsequence, $u_{k}\to u_{\infty}$ strongly in the $L^{p}$ sense and $\nabla^{M_{k}} u_{k} \to \nabla^{X}u_{\infty}$ weakly in the $L^{p}$-sense. Indeed, since $u_{k} \to u_{\infty}$ strongly in $W^{1,2}$ then, in particular, $\| u_{k} \|_{L^{2}}$ is a bounded sequence. Whence, recalling that $\| g_{k} \|_{L^{q}} \leq V^{1/q}$ for every $q>1$ and using the $L^{\infty}$-gradient estimates in \cite[Corollary 4.2]{ZZ-AdvMath} or \cite[Theorem 3.1]{Ji-JFA}, we deduce that $\| \nabla^{M_{k}} u_{k} \|_{L^{\infty}} \leq C$ for some uniform constant $C= C(m,D,K)>0$. This implies that $\| \nabla^{M_{k}}u_{k}\|_{L^{p}} \leq C_{1}$ for a suitable constant $C_{1}=C_{1}(p,m,D,K)>0$. Whence, using the Neumann-Poincar\`e inequality (or, again,\cite[Corollary 4.2]{ZZ-AdvMath})  yields that $\| u_{k}\|_{L^{p}} \leq C_{2} $ for some constant $C_{2}= C_{2}(p,m,D,K)>0$. It follows that $\{ u_{k} \} \subset W^{1,p}(M_{k})$ is a bounded sequence, as claimed. To conclude, we now apply the precompactness result contained in \cite[Theorem 4.9]{Ho-Crelle}.
\end{proof}


\end{document}